\documentclass[11pt]{article}

\usepackage{tikz}
\usepackage{subfigure}
\usepackage[english]{babel}
\usepackage{graphicx}

\usepackage[center]{caption2}
\usepackage{amsfonts,amssymb,amsmath,latexsym,amsthm}
\usepackage{multirow}

\def\ex{\mbox{ex}}
\def\sat{\mbox{sat}}
\def\ES{\mbox{ES}}

\newtheorem{thm}{Theorem}[section]

\newtheorem{lem}[thm]{Lemma}
\newtheorem{prop}{Proposition}
\newtheorem{obs}[thm]{Observation}
\newtheorem{conj}[thm]{Conjecture}
\newtheorem{prob}[thm]{Problem}

\newtheorem{claim}{Claim}
\newtheorem{fact}{Fact}

\begin{document}

\title{On the saturation spectrum of the unions of disjoint cycles
}
\author{Yue Ma$^a$\\
\small $^{a}$School of Mathematics and Statistics,\\
\small Nanjing University of Science and Technology,\\
\small Nanjing, Jiangsu 210094, China.\\
\small $^a$yma@njust.edu.cn
}

\date{}

\maketitle

\begin{abstract}
Let $G$ be a graph and $\mathcal{H}$ be a family of graphs. We say $G$ is $\mathcal{H}$-saturated if $G$ does not contain a copy of $H$ with $H\in\mathcal{H}$, but the addition of any edge $e\notin E(G)$ creates at least one copy of some $H\in\mathcal{H}$ within $G+e$. The saturation number of $\mathcal{H}$ is the minimum size of an $\mathcal{H}$-saturated graph on $n$ vertices, and the saturation spectrum of $\mathcal{H}$ is the set of all possible sizes of an $\mathcal{H}$-saturated graph on $n$ vertices. Let $k\mathcal{C}_{\ge 3}$ be the family of the unions of $k$ vertex-disjoint cycles. In this note, we completely determine the saturation number and the saturation spectrum of $k\mathcal{C}_{\ge 3}$ for $k=2$ and give some results for $k\ge 3$.
\end{abstract}


\section{Introduction}
All graphs considered in this paper are simple and finite. 
Let $G=(V,E)$ be a graph. We call $|V|$ the {\it order} of $G$ and $|E|$ the {\it size} of it. If $|V|=n$, we call $G$ an $n$-vertex graph. 
For positive integer $r$, let $K_r$ denote the complete graph on $r$ vertices; let $\overline{K_r}$ be the empty graph on $r$ vertices; let $P_r$ denote the path on $r$ vertices; when $r\ge 3$, let $C_r$ denote the cycle on $r$ vertices. For positive integer $r$ and $s$, let $K_{r,s}$ denote the bipartite complete graph whose two partition sets consist of $r$ vertices and $s$ vertices respectively.
For any $v\in V(G)$, we use $N_G(v)$ to denote the set of the neighbors of $v$ in $G$ and let $d_{G}(v)=|N_{G}(v)|$ be the {\it degree} of $v$ in $G$. Let $\delta(G)=\min_{v\in V(G)}d_{G}(v)$ be the minimum degree of $G$.
For any two graphs $G$ and $H$, let $G\cap H=(V(G)\cap V(H),E(G)\cap E(H))$ and $G\cup H=(V(G)\cup V(H),E(G)\cup E(H))$.
For a finite set $U$, let $K[U]$ be the complete graph on $U$. For any two sets $U_1$ and $U_2$, let $K[U_1,U_2]$ be a complete bipartite graph with two partition sets $U_1$ and $U_2$.
For a set $U\subseteq V(G)$, let $G[U]=G\cap K[U]$ be the subgraph induced by $U$. 
Write $G-U=G[V(G)\backslash U]$. 
For any two disjoint sets $U_1,U_2\subset V(G)$, let $G[U_1,U_2]=G\cap K[U_1,U_2]$.
For integers $a$ and $b$ with $a\le b$, let $[a,b]=\{a,a+1,a+2,\dots, b\}$. Write $\min A$ and $\max A$ for the minimum element and the maximum element of a bounded set $A$ of numbers, respectively.

Given a family $\mathcal{H}$ of graphs, a graph $G$ is said to be {\em $\mathcal{H}$-saturated} if $G$ does not contain a subgraph isomorphic to any member $H\in\mathcal{H}$, but  $G+e$ contains at least one copy of some $H\in\mathcal{H}$ for any edge $e\notin E(G)$.
The {\it Tur\'{a}n number} $\ex(n,\mathcal{H})$ of $\mathcal{H}$ is the maximum size of an $n$-vertex $\mathcal{H}$-saturated graph.
The {\it saturation number} $\sat(n,\mathcal{H})$ of $\mathcal{H}$ is the minimum size of an $n$-vertex $\mathcal{H}$-saturated graph. 
The problems to determine the Tur\'{a}n number and the saturation number of some given family $\mathcal{H}$ are known as the Tur\'{a}n problems and the saturation problems respectively. For saturation problems, one can refer to \cite{DS19} as a survey.
Another natural question is to determine all possible values $m$ between $\sat(n, \mathcal{H})$ and $\ex(n,\mathcal{H})$ such that there is an $n$-vertex $\mathcal{H}$-saturated graph with size $m$. We call the set of all possible sizes of an $\mathcal{H}$-saturated graph on $n$ vertices the {\it saturation spectrum} (or {\it edge spectrum}) of $\mathcal{H}$, denoted by $\ES(n,\mathcal{H})$. In other words,
$$\ES(n,\mathcal{H})=\{|E(G)|:G\mbox{ is an $n$-vertex }\mathcal{H}\mbox{-saturated graph}\}\mbox{.}$$
Clearly, $\ex(n,\mathcal{H})=\max \ES(n,\mathcal{H})$ and $\sat(n,\mathcal{H})=\min \ES(n,\mathcal{H})$. Thus, we have $\ES(n, \mathcal{H})\in[\sat(n, \mathcal{H}),\ex(n, \mathcal{H})]$.

The definition of saturation spectrums was firstly introduced by Barefoot et al.~\cite{K_3} in 1995. In that paper, they also determined $\ES(n,K_3)$. More generally, $\ES(n,K_r)$ for $r\ge 3$ was studied and given by Amin et al.~\cite{K_p} in 2013. Saturation spectrums of some other graphs are also determined, including $P_r$ ($3\le r\le 6$)~\cite{P}, $K_{1,r}$~\cite{star}, $K_4-e$~\cite{K_4-}, $C_5$~\cite{C_5} and so on. Recently, Lang et al.~\cite{LC} studied on the saturation spectrum of $\mathcal{C}_{\ge r}$, where $\mathcal{C}_{\ge r}$ denote the family of all cycles of size at least $r$. They determined $\ES(n, \mathcal{C}_{\ge r})$ for $r\in[3,6]$.

In this note, we focus on $k\mathcal{C}_{\ge 3}$ for $k\ge 2$, where $k\mathcal{C}_{\ge 3}$ is the family of the vertex-disjoint unions of $k$ members (may be identical) in $\mathcal{C}_{\ge 3}=\{C_3, C_4, \dots\}$. In 1962, Erd\"os and P\'osa~\cite{EP62} determined the Tur\'{a}n number for $k\mathcal{C}_{\ge 3}$ when $n\ge 24k$ and gave a conjecture for general $n$. This conjecture was later resolved by Justesen~\cite{J89}.
\begin{thm}[Justesen~\cite{J89}]
For integers $n$ and $k$ with $n\ge 3k\ge 6$, $\ex(n,k\mathcal{C}_{\ge 3})=\max\{\binom{3k-1}{2}+n-3k+1,(2k-1)n-2k^2+k\}$.
\end{thm}
In particular, for $k=2$, we have:
\begin{thm}[Justesen~\cite{J89}, see also Erd\"os, P\'osa~\cite{EP62}]\label{ex}
For $n\ge 6$, $\ex(n,2\mathcal{C}_{\ge 3})=3n-6$.
\end{thm} 
In this note, we firstly give the saturation version of Theorem~\ref{ex}.
\begin{thm}\label{sat}
For $n\ge 6$,
 \begin{equation*}
     	\sat(n,2\mathcal{C}_{\ge3})=\begin{cases}
     		10   & n=6;\\
     		n+5  & n\ge7.
     	\end{cases}
     \end{equation*}
\end{thm}
In addition, we determine the saturation spectrum of $2\mathcal{C}_{\ge3}$.
\begin{thm}\label{es}
(i) $\ES(6,2\mathcal{C}_{\ge3})=[\sat(6,2\mathcal{C}_{\ge3}), \ex(6,2\mathcal{C}_{\ge3})]=[10,12]$.\\
(ii) For integers $n\ge 7$ and $m\in[n+5,3n-6]$, there is an $n$-vertex $2\mathcal{C}_{\ge3}$ of size $m$ if and only if $m\in[n+5,2n-2]$ or $m-n$ is an even number. In other words,
$$\ES(n, 2\mathcal{C}_{\ge3})=[n+5,2n-2]\cup\{n+2t: t\in[3,n-3]\}\mbox{.}$$
\end{thm}
We also consider $sat(n,k\mathcal{C}_{\ge3})$ and $\ES(n,k\mathcal{C}_{\ge3})$ for $k\ge 3$ and give the following results.
\begin{thm}\label{k}
(i) For $k\ge 3$ and $n\ge 4k-3$, $\sat (n,k\mathcal{C}_{\ge3})\le n+6k-7$.\\
(ii) For $k\ge 5$ and $n\ge 4k^2-21k+37$, 
$$[n+6k-7,(2k-5)n-6k^2+21k-3]\subset \ES(n,k\mathcal{C}_{\ge3})\mbox{.}$$
\end{thm}
The rest of the article is arranged as follows. We give some constructions for $k\mathcal{C}_{\ge3}$-saturated graphs in Section 2 and finished the proof of Theorem~\ref{k}. In Section 3 and Section 4, we give the proof of Theorem~\ref{sat} and Theorem~\ref{es} respectively. In Section 5, there are some remarks and conjectures.

\section{Some $k\mathcal{C}_{\ge3}$-saturated constructions}

We give some definitions here. 
Let $G_1,G_2,\dots G_t$ be $t\ge 2$ graphs with $v_i\in V(G_i)$ ($i\in[1,t]$). 
Let $U(G_1,\dots,G_t)$ be the vertex-disjoint union of $G_1,\dots,G_t$. 
Let $U(G_1,\dots, G_t;v_1,\dots,v_t)$ be the union of $G_1,\dots,G_t$ where we identify $v_1,\dots,v_t$ as the same vertex. Note that if $G_i$ is a copy of a $K_2$ or a $2$-connected graph for some $i\in[1,t]$, then $G_i$ is a block of  $U(G_1,\dots, G_t;v_1,\dots,v_t)$.
Let $G_1\vee G_2$ be the graph obtained by adding all the edges between $V(G_1)$ and $V(G_2)$ in $U(G_1,G_2)$.
For any connected graph $G$ and $u,v\in V(G)$, let $P_G(u,v)\subseteq G$ be the shortest path with endpoints $u$ and $v$. If there are several choices for $P_G(u,v)$, we simply choose one.
For an edge $e=\{u,v\}\in E(G)$ ,we write $e=uv$ for short. For a path $P$ with $E(P)=\{v_1v_2,v_2v_3,\dots, v_{t-1}v_t\}$, we write $P=v_1v_2\dots v_t$ for short. For a cycle $C$ with $E(C)=\{v_1v_2,v_2v_3,\dots,v_{t-1}v_t,v_tv_1\}$, we write $C=v_1v_2\dots v_tv_1$ for short. Note that $uv$ for some $u,v\in V(G)$ can be an edge in $G$ or a path with only one edge in $G$. It depends on how we use it.

For any graph $G$, let $X(G)=\{v\in V(G):d_{G}(v)=n-1\}$ be the set of vertices in $G$ of degree $n-1$.
Let $n$ be an integer. For $t\in[1,n-2]$, let $S_{n,t}$ be the graph on $n$ vertices with $|X(S_{n,t})|=t$ and $S_{n,t}-X(S_{n,t})$ being an empty graph, i.e. $S_{n,t}\cong K_{t}\vee\overline{K_{n-t}}$. For $n\ge 5$, let $W_n$ be the graph on $n$ vertices with $|X(W_n)|=1$ and $W_n-X(W_n)$ being a cycle of size $n-1$, i.e. $W_n\cong K_1\vee C_{n-1}$. For $n\ge 5$ and $p\ge 1$, let $W_n^{+p}= G\cup H$ be a graph on $n+p$ vertices, where $G\cong W_n$, $H\cong P_{p+2}$ and the two vertices in $V(G)\cap V(H)$ are the two endpoints of $H$, one of which is also the only vertex in $X(G)$. More specifically, $W_n^{+p}$ is a graph with $V(W_n^{+p})=[0,n+p-1]$ and
$$E(W_n^{+p})=\bigcup_{i=1}^{n-1}\{0i\}\cup\bigcup_{i=1}^{n-2}\{i(i+1)\}\cup \{(n-1)1, 0n, (n+p-1)1\}\cup\bigcup_{i=n}^{n+p-2}\{i(i+1)\}\mbox{.}$$
For convenience, we call $0$ the center vertex of $W_n^{+p}$. We also call the subgraph induced by $[1,n-1]$ (resp. $[n,n-p-1]$) the cycle part (resp. path part) of $W_n^{+p}$, and use $CP(W_n^{+p})$ (resp. $PP(W_n^{+p})$) to denote it.

\begin{figure}[h]
\centering
\subfigure[$S_{6,3}$]{\includegraphics[width=1.1in]{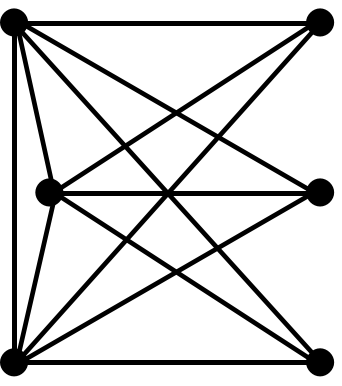}}
\hspace{0.5in}
\subfigure[$W_7$]{\includegraphics[width=1.2in]{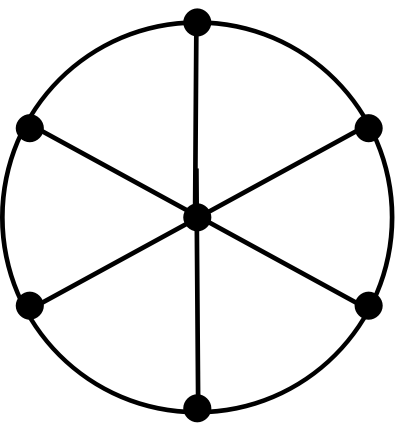}}
\hspace{0.5in}
\subfigure[$W_{6}^{+2}$]{\includegraphics[width=1.2in]{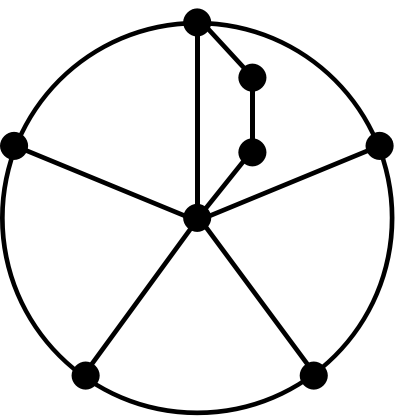}}
\caption{Some examples for the structures mentioned}\label{bl}
\end{figure}

One can simply check the following observations..

\begin{obs}\label{block}
(i) For $n\ge 3k\ge 3$, $S_{n,2k-1}$ is a $k\mathcal{C}_{\ge3}$-saturated graph. Moreover, $|E(S_{n,2k-1})|=(2k-1)n-2k^2+k$.\\
(ii) For $n\ge 6$ and $p\ge 1$, $W_n$ and $W_n^{+p}$ are $2\mathcal{C}_{\ge3}$-saturated graphs. Moreover, $|E(W_n)|=2n-2$ and $|E(W_{n}^{+p})|=2(n+p)-p-1$.\\
(iii) Let $G\cong K_{3k-1}$ with $k\ge 2$. Pick $u\in V(G)$. For any $v\in V(G)$, there is a path $P\subset G$ with the following properties: (a) the two endpoints of $P$ is $v$ and $u$; (b) $G-V(P)$ contains $k-1$ disjoint cycles.\\
(iv) Let $n\ge 3k$ and $G\cong S_{n,2k-1}$. Pick $u\in V(G)\backslash X(G)$. For any $v\in V(G)$, there is a path $P\subset G$ with the following properties: (a) the two endpoints of $P$ is $v$ and $u$; (b) $G-V(P)$ contains $k-1$ disjoint cycles.\\
(v) Let $n\ge 6$ and $G\cong W_n^{+p}$. Pick a vertex $u$ from the path part of $G$. For any $v\in V(G)$, there is a path $P\subset G$ with the following properties: (a) the two endpoints of $P$ is $v$ and $u$; (b) $G-V(P)$ contains a cycle.
\end{obs}
\begin{proof}
Since (i) and (ii) can be checked by simple discussions, we only prove (iii), (iv) and (v). Also, for (iii), we can simply pick the single edge $uv$ as $P$ (or the single vertex $u$ if $u=v$) and complete the proof.\\
For (iv), we pick $P=P_{G}(u,v)$. If $v=u$, then $P\cong K_1$ and $G-V(P)\cong S_{n-1,2k-1}$; if $v\in X(G)$, then $P\cong K_2$ and $G-V(P)=S_{n-2,2k-2}$; else $v\in V(G)\backslash(X(G)\cup\{u\})$ and $P\cong P_3$, then $G-V(P)=S_{n-3,2k-2}$. In all these cases, we can find $k-1$ vertex-disjoint cycles in $G-V(P)$.\\
For (v), if $v$ is the center vertex $c$ or a vertex in the path part $PP(G)$, we pick $P=P_{G[\{c\}\cup V(PP(G))]}(u,v)$. In this case, $CP(G)\subseteq G-V(P)$, which is a cycle in $G-V(P)$. Otherwise, $v$ is a vertex in $CP(G)$, we pick $P=P_{CP(G)}(u,v)$. Since $n\ge 6$, $CP(G)-V(P)$ is a path with at least $\lceil\frac{6-1}{2}\rceil-1=2$ vertices, which means we can pick an edge $e\in E(CP(G)-V(P))$. This edge, together with the center vertex of $G$, induce a cycle $C_3$ in $G-V(P)$.
\end{proof}

Let $\mathcal{G}$ be the family of all connected graphs whose blocks are copies from $\{K_{3i-1}:i\ge 1\}\cup\{W_a^{+p}:a\ge 6\mbox{, }p\ge 1\}\cup\{S_{b,2j-1}:b\ge 3j\ge 6\}$. Let $G\in \mathcal{G}$ and let $k_i$, $w_{a,p}$ and $s_{b,j}$ be the numbers of the blocks isomorphic to $K_{3i-1}$, $W_{a}^{+p}$ and $S_{b,2j-1}$ for all possible $i,a,p,b,j$ in $G$ respectively.
we now give the constructions as in the following lemma.
\begin{lem}\label{construct}
For $k\ge 2$, we can pick a proper $G\in \mathcal{G}$ which is $k\mathcal{C}_{\ge3}$-saturated, if $G$ is not a complete graph and
$$k=1+\sum_{i\ge 1}(i-1)k_i+\sum_{a\ge 6,p\ge 1}w_{a,p}+\sum_{b\ge 3j\ge 6}(j-1)s_{b,j}\mbox{.}$$
\end{lem}
\begin{proof}
Note that by (i) and (ii) in Observation~\ref{block}, the sum of the numbers of the vertex-disjoint cycles in all the blocks in $G$ is at most $\sum_{i\ge 1}(i-1)k_i+\sum_{a\ge 6,p\ge 1}w_{a,p}+\sum_{j\ge 2}(j-1)s_{b,j}=k-1<k$. Since any cycle in $G$ must be contained in a block, we see that $G$ does not contain $k$ disjoint cycles.\\ 
Now it remains to pick a proper $G$ and prove that $G+e$ contains $k$ disjoint cycles for any $e\notin E(G)$.
Let $G_1, G_2, \dots, G_m$ be all the blocks in $G$ and let $x_t$ be the maximum number of disjoint cycles in $G_t$ for any $t\in[1,m]$. Note that $\sum_{t=1}^mx_t=k-1$. For $t\in[1,m]$, we pick $v_t\in V(G_t)$ with the following rules:\\ 
(a) if $G_t\cong K_{3i-1}$ for some $i\ge 1$, pick $v_t$ arbitrarily; \\
(b) if $G_t\cong W_a^{+p}$ for some $a\ge 6$ and $p\ge 1$, pick $v_t$ be a vertex from the path part of $G_t$;\\
(c) if $G_t\cong S_{b,2j-1}$ for some $b\ge 3j\ge 6$, pick $v_t\in V(G_t)\backslash X(G_t)$.\\
Now we pick $G=U(G_1,\dots,G_m;v_1,\dots,v_m)$. We will prove that $G+u_1u_2$ contains $k$ disjoint cycles for any $u_1u_2\notin E(G)$. Let $u_1\in V(G_{t_1})$ and $u_2\in V(G_{t_2})$.
If $t_1=t_2$, then by Observation~\ref{block} (iii), (iv) and (v), we can pick a path $P\subseteq G_{t_1}$ from $u_1$ to $u_2$ with $G_{t_1}-V(P)$ still containing $x_{t_1}$ disjoint cycles. By the same observation, for any $s\in[1,m]\backslash\{t_1\}$, $G_s-\{v_s\}$ contains $x_s$ disjoint cycles. It is easy to see that all these cycles, together with $u_1u_2\cup P$, are $k$ disjoint cycles in $G+u_1u_2$.
Now suppose $t_1\neq t_2$. Similarly by Observation~\ref{block} (iii), (iv) and (v), for any $\ell\in\{1,2\}$, we can pick a path $P\subseteq G_{t_\ell}$ from $u_\ell$ to $v_{t_1}=v_{t_2}$ with $G_{t_\ell}-V(P)$ still containing $x_{t_\ell}$ disjoint cycles. Also, $G_s-\{v_s\}$ contains $x_s$ disjoint cycles for any $s\in[1,m]\backslash\{t_1, t_2\}$. All these cycles, together with the cycle $P_1\cup u_1u_2\cup P_2$, are $k$ disjoint cycles in $G+u_1u_2$.
\end{proof}

\begin{figure}[h]
\centering
\includegraphics[width=3.5in]{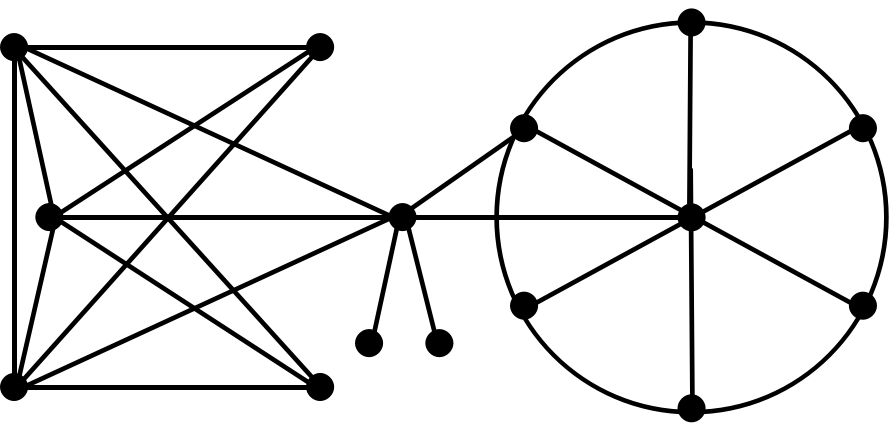}
\caption{An example for a $3\mathcal{C}_{\ge 3}$-saturated graph by Lemma~\ref{construct}}\label{c1}
\end{figure}

By Lemma~\ref{construct} and some calculations, we give the following theorem.
\begin{thm}\label{cal}
(i) For $k\ge 2$ and $n\ge \max\{3k,4k-3\}$, $\sat (n,k\mathcal{C}_{\ge3})\le n+6k-7$.\\
(ii) For $k\ge 2$ and $n\ge 4k-1$, $[n+6k-7,2n+2k-6]\subset \ES(n,k\mathcal{C}_{\ge3})$. \\
(iii) For $k\ge 3$ and $n\ge5k-1$, $[n+6k-7,3n-3k-3]\subset \ES(n,k\mathcal{C}_{\ge3})$.\\
(iv) For $k\ge 5$ and $n\ge 5k-4$, $[2n+4k^2-24k+34,(2k-5)n-6k^2+21k-3]\subset \ES(n,k\mathcal{C}_{\ge3})$.\\
(v) For $k\ge 2$ and $n\ge 3k$, $\{n-2k^2+k+(2k-2)t: t\in[3k,n]\}\subset \ES(n,k\mathcal{C}_{\ge3})$.
\end{thm}
\begin{proof}
Pick $G$ in Lemma~\ref{construct} with $|V(G)|=n$. 
Throughout the proof, we let $k_i=0$ for $i\ge 3$ and $w_{a,p}=0$ for $p\ge 2$.\\
For the proof of (i) and (ii), let $s_{b,j}=0$ for all possible $b$ and $j$ in addition. Then by Lemma~\ref{construct}, we have 
\begin{equation}\label{e1}
\sum_{a\ge 6}w_{a,1}=k-k_2-1\mbox{.}
\end{equation}
Also, by simple calculations,
\begin{equation}\label{e2}
n=1+k_1+4k_2+\sum_{a\ge 6}aw_{a,1}\mbox{,}
\end{equation}
\begin{equation}\label{e3}
|E(G)|=k_1+10k_2+2\sum_{a\ge 6}aw_{a,1}=2n-2+2k_2-k_1\mbox{.}
\end{equation}
We firstly pick $w_{a,1}=0$ for any $a\ge 6$. 
Then by (\ref{e1}), $k_2=k-1$ and thus by (\ref{e2}) and (\ref{e3}) $n=1+k_1+4k_2=4k-3+k_1$ and $|E(G)|=(2n-2)-2(k-1)-(n-4k+3)=n+6k-7$. Therefore, by choosing different $k_1\ge 0$ when $k\ge 3$ and $k_1\ge 1$ when $k=2$ (since $G\cong K_5$ is a complete graph when $k_1=0$), we can make $G$ be any $k\mathcal{C}_{\ge 3}$-saturated graph with $n\ge\max\{3k,4k-3\}$ and $n+6k-7$ edges. This completes (i).\\
Now we pick $w_{a_0,1}=1$ for some $a_0\ge 6$ and  $w_{a,1}=0$ for any $a\ge 6$ and $a\neq a_0$. 
By (\ref{e1}), $k_2=k-2$. By (\ref{e2}), $n=a_0+k_1+4k-7\ge 4k-1$. It is easy to check that, for any fixed $n\ge 4k-1$ and $k_1\in[0, n-4k+1]$, we always have a proper value for $a_0\ge 6$.
Thus by (\ref{e3}), $|E(G)|=2n+2k-6-k_1$ can be any value in $[n+6k-7,2n+2k-6]$ for any $n\ge 4k-1$, which completes the proof of (ii).\\
For the proof of (iii) and (iv), let $k_1=k_2=0$, $w_{a_0,1}=1$ for some $a_0\ge 6$ and  $w_{a,1}=0$ for any $a\ge 6$ and $a\neq a_0$. Then Similarly as the previous proof, we have the follwing equations,
\begin{equation}\label{e4}
k-2=\sum_{b\ge 3j\ge 6}(j-1)s_{b,j}\mbox{,}
\end{equation} 
\begin{equation}\label{e5}
n=1+a_0+\sum_{b\ge 3j\ge 6}(b-1)s_{b,j}\mbox{,}
\end{equation}
\begin{equation}\label{e6}
|E(G)|=2a_0+\sum_{b\ge 3j\ge 6}(2j-1)(b-j)s_{b,j}\mbox{.}
\end{equation}
For (iii), let $s_{b,j}=0$ for any $j\neq 2$ and let $s_{6,2}\in\{k-3,k-2\}$. Then by (\ref{e4}), there exists at most one integer $b\ge 7$ with $s_{b,2}\ge 1$ and we must have $s_{6,2}=k-3$ and $s_{b,2}=1$ when it exists. By (\ref{e4}), (\ref{e5}) and (\ref{e6}), one can check that
\begin{equation}\label{e12}
|E(G)|=3n-3k+3-a_0\mbox{.}
\end{equation}
If $s_{6,2}=k-3$, suppose $b_0\ge 7$ is the only integer $b\ge 7$ with $s_{b,2}=1$.
By (\ref{e5}), $6\le a_0=n-5k+14-b_0\le n-5k+7$. Clearly, for fixed $n$ and $k$, we can arbitrarily choose the value of $a_0$ from $[6, n-5k+7]$ since $n\ge 5k-1$. So by (\ref{e12}), $|E(G)|$ can be any of the number in $[2n+2k-4,3n-3k-3]$.
If $s_{6,2}=k-2$, by (\ref{e5}), $a_0=n-5k+8\ge 6$. This implies $|E(G)|=2n-2k-5$ by (\ref{e12}). Therefore,  $|E(G)|$ can be any of the number in $[2n+2k-5,3n-3k-3]$. 
Thus by (ii), we are done for (iii).\\
For (iv), let $s_{b_1,k-2}=1$ for some $b_1\ge 3k-6$. Then by (\ref{e4}), we can see that $s_{b_2,2}=1$ for some $b_2\ge 6$ and $s_{b,j}=0$ for any other possible $(b,j)$. By (\ref{e4}), (\ref{e5}) and (\ref{e6}), one can check that
\begin{equation}\label{e7}
|E(G)|=(2k-5)n-2k^2+11k-21-((2k-8)(a_0+b_2)+a_0)\mbox{.}
\end{equation}
Note that by (\ref{e5}), $12\le a_0+b_2=n+1-b_1\le n+7-3k$. It is easy to see that, for any fixed $n$ and $k$, we can arbitrarily choose the value of $a_0+b_2$ from $[12,n+7-3k]$ and then choose the value of $a_0$ from $[6,a_0+b_2-6]$. 
Now let $A$ be all the values can be chosen by $(2k-8)(a_0+b_2)+a_0$ for fixed $n$ and $k$,
and let $A(c)$ be all the values can be chosen by $(2k-8)(a_0+b_2)+a_0$ for fixed $n$, $k$ and $a_0+b_2=c$. Note that since $k\ge 5$ and $n\ge 5k-4$, it holds $2k+3\in[12,n+7-3k]$, so we can pick $c\in[2k+3,n+7-3k]$. Then $[(2k-8)c+6,(2k-8)(c+1)+5]\subset[(2k-8)c+6,(2k-7)c-6]\subseteq A(c)$. Therefore, 
$$[(2k-8)(2k+3)+6,(2k-7)(n+7-3k)-6]\subseteq\bigcup_{c\in[2k+3,n+7-3k]}A(c)\subseteq A\mbox{.}$$
Thus, 
\begin{equation}\label{e8}
[4k^2-10k-18,(2k-7)n-6k^2+35k-55]\subseteq A\mbox{.}
\end{equation} 
Let $D$ be the set of all the values can be chosen by $|E(G)|$.
Clearly, by (\ref{e7}) and (\ref{e8}), 
\begin{equation}\label{e9}
[2n+4k^2-24k+34,(2k-5)n-6k^2+21k-3]\subseteq D\mbox{.}
\end{equation}
Thus, we complete the proof of (iv).\\
Now we give the proof of (v). Let $k_2=0$ and $w_{a,1}=0$ for any $a\ge 6$ in this case. Let $(b_0,j_0)$ with $b_0\ge 3j_0\ge 6$ be the only pair of $(b,j)$ with $s_{b,j}\ge 0$ and let $s_{b_0,j_0}=1$. Then similarly, we get $j_0=k$, $n=k_1+b_0$ and
\begin{equation}\label{e13}
|E(G)|=k_1+(2j_0-1)(b_0-j_0)=(2k-2)b_0+n-2k^2+k\mbox{.}
\end{equation}
 Note that $3k=3j_0\le b_0=n-k_1\le n$. Apparently, we can arbitrarily choose the value of $b_0$ from $[3k,n]$ since $n\ge 3k$. Therefore, $|E(G)|$ can be any number in $\{n-2k^2+k+(2k-2)t: t\in[3k,n]\}$. This completes the proof.
\end{proof}
Note that Theorem~\ref{k} (i) can be deduced by Theorem~\ref{cal} (i), and Theorem~\ref{k} (ii)  follows from Theorem~\ref{cal} (ii) and (iii), since when $n\ge 4k^2-21k+37>5k-1$, it holds $3n-3k-3\ge 2n+4k^2-24k+34$. Thus, Theorem~\ref{k} is done.

\section{Saturation number of $2\mathcal{C}_{\ge3}$}
For any graph $G$, a {\it subdivision} of $G$ is a graph obtained from $G$ be replacing the edges of $G$ with internally disjoint paths of sizes at least $1$. We use $\mathcal{S}(G)$ to denote the family of all the subdivisions of $G$. 
We have the following lemma.
\begin{lem}\label{sub}
Let $G$ be a graph and $H\in\mathcal{S}(G)$, then $|E(H)|-|V(H)|=|E(G)|-|V(G)|$. Moreover, if $H$ is $k\mathcal{C}_{\ge3}$-saturated for some $k\ge 2$, then $G\cong K_{3k-1}$ or $G$ is also $k\mathcal{C}_{\ge3}$-saturated.
\end{lem}
\begin{proof}
Let $uv\in E(G)$. Define $S_{uv}(G)$ be the graph obtained from $G$ by replacing the edge $uv$ with a path $uwv$ for a new vertex $w\notin V(G)$. Clearly, $S_{uv}(G)\in\mathcal{S}(G)$, and we can get $H$ from $G$ by iteratively using the operation $S_e(\cdot)$ for different edges $e$. Thus, we only need to prove for $H=S_{uv}(G)$.
Suppose $H=S_{uv}(G)$ and let $V(H)\backslash V(G)=\{w\}$. 
It is easy to see that $|E(H)|-|V(H)|=(|E(G)|+1)-(|V(G)|+1)=|E(G)|-|V(G)|$, so we only need to prove for the second part.
If $G$ is a complete graph, one can easily check that $H$ is $k\mathcal{C}_{\ge3}$-saturated if and only if $|V(G)|=3k-1$, i.e. $G\cong K_{3k-1}$. Now suppose $G$ is not complete. Clearly, if $G$ contains $k$ disjoint cycles, then $H$ must contain the subdivisions of these cycles, which are also $k$ disjoint cycles. This contradicts the fact that $H$ is $k\mathcal{C}_{\ge3}$-saturated. Hence,  $G$ does not contain $k$ disjoint cycles.
Suppose the contrary that $G$ is not  $k\mathcal{C}_{\ge3}$-saturated.
Then there exists a pair of vertices $ab\notin E(G)$ with $G+ab$ not containing $k$ disjoint cycles.
Since $H$ is $k\mathcal{C}_{\ge3}$-saturated, there exist $k$ disjoint cycles $C^1,\dots, C^k$ in $H+ab$. If $w\notin C^i$ for any $i\in[1,k]$, then $C^1,\dots, C^k$ are also $k$ disjoint cycles in $G+ab$, a contradiction. Thus, we can suppose $w\in C^1$. Note that $uv\notin E(H)$. So we can replace the path $uwv$ by $uv$ in $H$ and obtain a cycle $C^{1-}$ in $G$. Then $C^{1-},C^2,\dots, C^k$ are $k$ disjoint cycles in $G+ab$, a contradiction.
\end{proof}
We use a trivial block to denote a block isomorphic to $K_2$.
For any graph $G$, let $M_0(G)$ be the graph obtained from $G$ by iteratively deleting the vertices of degree one until there is no vertex of degree one in the graph. Let $M(G)$ be the graph of minimum size with $M_0(G)\in\mathcal{S}(M(G))$. 
For any integer $i\ge 0$, let $D_i(G)=\{v\in V(G): d_{G}(v)=i\}$.
We have the following proposition for $M(G)$.
\begin{prop}\label{mg}
Let $G$ be a connected graph. The following hold.\\
(i) $|E(M(G))|-|V(M(G))|=|E(G)|-|V(G)|$.\\
(ii) The number of non-trivial blocks in $G$, $M_0(G)$ and $M(G)$ are the same. In particular, If $G$ contains exact one non-trivial block, then $M(G)$ is $2$-connected.\\
(iii) If $G$ contains some non-trivial blocks, then $\delta(M(G))\ge 2$. Moreover, if $v\in D_2(M(G))$ and $N_{M(G)}(v)=\{u_1,u_2\}$, then $u_1u_2\in E(M(G))$.\\
(iv) If $G$ is $k\mathcal{C}_{\ge 3}$-saturated for some $k\ge 2$, then either $M(G)$ is also $k\mathcal{C}_{\ge 3}$-saturated, or $M(G)\cong K_{3k-1}$.
\end{prop}  
\begin{proof}
For (i), note that every time we delete a vertex of degree one, we must delete exact one edge. This gives $|E(M_0(G))|-|V(M_0(G))|=|E(G)|-|V(G)|$. Now since $M_0(G)\in\mathcal{S}(M(G))$, by Lemma~\ref{sub}, $|E(M(G))|-|V(M(G))|=|E(M_0(G))|-|V(M_0(G))|=|E(G)|-|V(G)|$. We are done.\\
For (ii), since any non-trivial block $B$ of $G$ has $\delta(B)\ge 2$, all the non-trivial blocks in $G$ remain the same in $M_0(G)$. Also, by the definition of subdivision, it is easy to see that there is a bijection from the the non-trivial blocks in $M_0(G)$ to the the non-trivial blocks in $M(G)$. This completes the proof of the first part of (ii). 
If $G$ contains exact one non-trivial block $B$, then the trivial blocks in $G$ induce a forest with each component (which is a non-empty tree) sharing exact one vertex with $B$. Since $M_0(G)$ is a connected subgraph of $G$ and $B\subseteq M_0(G)$, this property still holds in $M_0(G)$ if there are still some trivial blocks in $M_0(G)$. Note that every component of the forest in $M_0(G)$ has at least one vertex of degree $1$ other than the shared vertex with $B$. This implies $D_1(M_0(G))\neq\emptyset$, which is a contradiction by the definition of $M_0(G)$. Thus, there is no trivial block in $M_0(G)$. Therefore, $M_0(G)=B$ is $2$-connected. This then makes $M(G)$ $2$-connected, too.\\
For (iii), note that $D_1(M(G))=D_1(M_0(G))=\emptyset$. If  $\delta(M(G))<2$, then $D_0(M(G))\neq\emptyset$. Since $M(G)$ is connected, this means $M(G)\cong K_1$. However, by (ii), $M(G)$ has non-trivial blocks, which is a contradiction. Thus, $\delta(M(G))\ge 2$. 
Moreover, if $v\in D_2(M(G))$ and $N_{M(G)}(v)=\{u_1,u_2\}$ with $u_1u_2\notin E(M(G))$, let $H=(M(G)-\{v\})+u_1u_2$. Clearly, $G\in\mathcal{S}(M(G))\subset\mathcal{S}(H)$ and $H$ has a smaller size than $M(G)$. This is a contradiction by the minimality of $M(G)$.\\
Now we show (iv) to complete the proof. 
Since $k\ge 2$, $G$ contains at least one non-trivial block (since it contains at least $k-1$ disjoint cycles). So if $M_0(G)$ is a complete graph, $M_0(G)$ must be the only non-trivial block of $G$ by (ii). Also, $M_0(G)\cong K_t$ does not contain $k$ disjoint cycles, so it holds $t\le 3k-1$. Since $G$ is connected and not complete, pick $x\in V(G)\backslash V(M_0(G))$ with $xy\in E(G)$ for some $y\in V(M_0(G))$. By definitions, $xz\notin E(G)$ for any $z\in V(M_0(G))\backslash\{y\}$. Note that $H=(G+xz)[V(M_0(G))\cup\{z\}]$ is the only non-trivial block in $G+xz$, which must contains $k$ disjoint cycles. However, if $t\le 3k-2$, then $V(H)\le 3k-1$, which implies that $H$ cannot contain $k$ disjoint cycles. This means $t=3k-1$ and $M(G)=M_0(G)\cong K_{3k-1}$. We are done.
Now suppose $M_0(G)$ is not complete and let $uv$ be a pair of vertices in $V(M_0(G))$ with $uv\notin E(M_0(G))=E(G[V(M_0(G))])$. Since $G$ is $k\mathcal{C}_{\ge 3}$-saturated, there exist $k$ disjoint cycles in $G+uv$.  Similarly as in the proof of (ii), all the non-trivial blocks in $G$ remain the same in $M_0(G)$, which means all the edges in $(G+uv)\backslash M_0(G)$ are trivial blocks. Hence the $k$ disjoint cycles in $G+uv$ do not use the edges in $(G+uv)\backslash M_0(G)$ since the edges in $(G+uv)\backslash M_0(G)$ are still trivial blocks in $G+uv$. Thus, these $k$ disjoint cycles are also in $M_0(G)$, which implies that $M_0(G)$ is a $k\mathcal{C}_{\ge 3}$-saturated graph. By Lemma~\ref{sub}, either $M(G)$ is also $k\mathcal{C}_{\ge 3}$-saturated, or $M(G)\cong K_{3k-1}$. This completes the proof.
\end{proof}
We call any $2$-connected graph $G$ with $M(G)=G$ a good graph. Note that for any graph $H$ with exact one non-trivial block, $M(H)$ is a good graph by Propostion~\ref{mg} (iv).

Now we focus on $2\mathcal{C}_{\ge 3}$-saturated graph.
To make the proof more clear, we firstly mention the following two facts.
\begin{fact}\label{f1}
Let $G$ be a $2\mathcal{C}_{\ge 3}$-saturated graph and $uv$ be a pair of vertices with $uv\notin E(G)$. Then there exist a path $P$ from $u$ to $v$ on at least $3$ vertices and a cycle $C$ in $G$ with $V(C)\cap V(P)=\emptyset$. In particular, $G$ is connected.
\end{fact}
\begin{fact}\label{f2}
Let $G$ be a $2\mathcal{C}_{\ge 3}$-saturated graph and $C\subset G$ be a cycle. If $|V(C)|<|V(G)|$, then $G-V(C)$ is a forest.
\end{fact}
Here are three properties for a $2\mathcal{C}_{\ge 3}$-saturated graph.
\begin{lem}\label{pre}
Let $G$ be a $2\mathcal{C}_{\ge 3}$-saturated graph. Then the following hold.\\
(i) $G$ is a connected graph with exact one non-trivial block.\\
(ii) If $G$ is good, then there exists a cycle of size at most $4$ in $G$.\\
(iii) Let $C$ be a cycle in $G$ with $|V(C)|<|V(G)|=n$. If $\delta(G)\ge 3$, then 
$$|E(G)|\ge n+2d_0+\frac{3}{2}d_1+d_2+\frac{1}{2}d_{\ge 3}\mbox{,}$$
where $d_i=|D_{i}(G-V(C))|$ for every $i\ge 0$ and $d_{\ge 3}=\sum_{j\ge 3}d_j$.
Moreover, the equality holds if and only if $d_i=0$ for $i>3$ and $|N_{G}(v)\cap V(C)|=3-i$ for any $v\in D_{i}(G-V(C))$ and $i\in[1,3]$.
\end{lem}
\begin{proof}
By Fact~\ref{f1}, $G$ is connected and $G$ contains a cycle. So $G$ contains at least one non-trivial block. 
For (i), suppose the contrary that $G$ contains at least two non-trivial blocks $B_1$ and $B_2$. If $V(B_1)\cap V(B_2)=\emptyset$. Then we can pick two disjoint cycles $C^1$ and $C^2$ in $G$ with $C^i\subset B_i$ ($i=1,2$), which is a contradiction. Thus, $B_1$ and $B_2$ share a cut vertex, say $x$. 
We claim that any cycle $C$ in $G$ must have $x\in V(C)$. Otherwise, suppose there is a cycle $C^0$ with $x\notin V(C^0)$. Note that $V(C^0)$ is disjoint with at least one of $V(B_1)$ and $V(B_2)$, or $B_1$ and $B_2$ would be contained in a larger block. Without loss of generality, suppose $V(C^0)\cap V(B_1)=\emptyset$. We can pick a cycle $C^1\subset B_1$, which is disjoint with $C^0$, a contradiction. Hence, it holds that $x\in V(C)$ for any cycle $C$ in $G$. 
Pick $v_i\in V(B_i)\backslash\{x\}$ for $i=1,2$, then $v_1v_2\notin E(G)$. Since $x$ is the cut vertex of $B_1$ and $B_2$, any path $P$ from $v_1$ to $v_2$ must have $x\in V(P)$. However, any cycle $C$ in $G$ also has $x\in V(C)$. This is a contradiction by Fact~\ref{f1}.\\
Now we give the proof of (ii) and suppose $G$ is good. 
By Proposition~\ref{mg} (iii), if there exists a vertex $v\in D_2(G)$, then $u_1u_2\in E(G)$ for the two neighbors $u_1$ and $u_2$ of $v$. Then $G[\{u_1,u_2,v\}]=vu_1u_2v\cong C_3$, and we are done. So we may suppose $D_2(G)=\emptyset$ and then $\delta(G)\ge3$. Let $C$ be the shortest cycle in $G$. Since $C$ is a hole (i.e. an induced cycle), if $|V(C)|=|V(G)|$, then $G=C$, which is clearly not a $2\mathcal{C}_{\ge 3}$-saturated graph. Hence, by Fact~\ref{f2}, $G-V(C)$ is a forest and so we can pick a vertex $v\in D_1(G-V(C))\cup D_0(G-V(C))$. Then $|N_{G}(v)\cap V(C)|\ge \delta(G)-1\ge 2$. Pick $\{u_1,u_2\}\subseteq N_{G}(v)\cap V(C)$. Then $vu_1\cup P_{C}(u_1,u_2)\cup u_2v$ is a cycle of size at most $\lfloor|E(C)|\rfloor+2$ (recall that $P_{C}(u_1,u_2)$ is the shortest path from $u_1$ to $u_2$ in $C$). By the minimality of $C$, we have $|E(C)|\le \lfloor|E(C)|\rfloor+2$, which implies $|E(C)|\le 4$.\\
For (iii), suppose $C\cong C_r$ for some $r\ge 3$. By Handshaking Lemma, $|E(G-V(C))|=\frac{1}{2}\sum_{i\ge 0}id_i\ge\frac{1}{2}d_1+d_2+\frac{3}{2}d_{\ge 3}$. Also, for any $i\in[0,2]$ and vertex $v\in D_i(G-V(C))$, since $d_{G}(v)\ge 3$, it holds $|N_{G}(v)\cap V(C)|\ge 3-i$. Thus, 
$$|E(G[V(C),V(G)\backslash V(C)])|=\sum_{v\in V(G)\backslash V(C)}|N_{G}(v)\cap V(C)|\ge 3d_0+2d_1+d_2\mbox{.}$$
Therefore, it holds
\begin{eqnarray*}
|E(G)|&=&|E(C)|+|E(G[V(C),V(G)\backslash V(C)])|+|E(G-V(C))|\\
&\ge & r+3d_0+\frac{5}{2}d_1+2d_2+\frac{3}{2}d_3\\
&=& n+2d_0+\frac{3}{2}d_1+d_2+\frac{1}{2}d_{\ge 3}\mbox{,}
\end{eqnarray*}
since $r+d_0+d_1+d_2+d_{\ge 3}=n$. It is easy to check the condition when equality holds.
\end{proof}
Now we give the proof of Theorem~\ref{sat} when $n=6$ and show that a $6$-vertex $2\mathcal{C}_{\ge 3}$-saturated graph with $10$ edges must be a copy of $W_6$.
\begin{lem}\label{n6}
$\sat(6,2\mathcal{C}_{\ge 3})=10$. Moreover, the only $6$-vertex $2\mathcal{C}_{\ge 3}$-saturated graph with $10$ edges is $W_6$.
\end{lem}
\begin{proof}
Consider a $6$-vertex $2\mathcal{C}_{\ge 3}$-saturated graph $G$ with $|E(G)|\le 10$. By Proposition~\ref{mg} (iv), $M(G)\cong K_5$ or $M(G)$ is also $2\mathcal{C}_{\ge 3}$-saturated.
If $M(G)\cong K_5$, then by Proposition~\ref{mg} (i), $|E(G)|=10+6-5=11$, a contradiction. So $M(G)$ is $2\mathcal{C}_{\ge 3}$-saturated. Hence $|V(M(G))|\ge 2\times 3=6$, which implies $G=M(G)$. By Proposition~\ref{mg} (ii) and Lemma~\ref{pre} (i), $G$ is good.\\
If $D_2(G)\neq \emptyset$, then by Proposition~\ref{mg} (iii), we can pick a cycle $C\cong C_3$ in $G$ containing a vertex of degree $2$. Suppose $C=vu_1u_2v$ and $d_{G}(v)=2$. Let $A=V(G)\backslash V(C)$.
Claim that $G[\{u_1,u_2\},A]=K[\{u_1,u_2\},A]$. Otherwise, say $u_1a\notin E(G)$ with $a\in A$. By Fact~\ref{f1} and $|V(G)|=6$, we can find a path $P=u_1ba$ for some $b\in V(G)$ and a cycle $C'\cong C_3$ with $V(C')=V(G)\backslash V(P)$. If $b\in V(C)$, then $b=u_2$ since $b\neq u_1$ and $va\notin E(G)$. Then $v\in V(C')$, which cannot holds since $N_{G}(v)=\{u_1,u_2\}\subset V(P)$. If $b\in A$, similarly, $v\in V(C')$, a contradiction, too. Thus, the claim holds.
Note that $|E(G[A])|=|E(G)|-|E(C)|-|E(G[V(C),A])|\le 10-3-2\times 3=1$. If $|E(G[A])|=0$, then $G\cong S_{6,2}$; if $|E(G[A])|=1$, $G$ can be obtained by adding an edge in $S_{6,2}$. In both cases, one can check that $G$ is not $2\mathcal{C}_{\ge 3}$-saturated.\\
Therefore, $D_2(G)=\emptyset$ and $\delta(G)\ge 3$. Since $G$ is not complete and $|V(G)|=6$, by Fact~\ref{f1}, there is a path $P=p_1p_2p_3$ and a cycle $C=c_1c_2c_3c_1$ with $V(C)=V(G)\backslash V(P)$ and $G[V(P)]=P$. By Lemma~\ref{pre} (iii), $10\ge |E(G)|\ge6+\frac{3}{2}\times 2+1=10$, which means the equality holds. So $|E(G)|=10$ and $|N_{G}(p_i)\cap V(C)|=2$ for $i=1,3$ and $|N_{G}(p_2)\cap V(C)|=1$.
Without loss of generality, suppose $N_{G}(p_2)\cap V(C)=\{c_2\}$.\\
If $N_{G}(p_1)\cap V(C)=\{c_1,c_3\}$, then $c_2\notin N_{G}(p_3)\cap V(C)$. Otherwise, $p_1c_1c_3p_1$ and $p_3c_2p_2p_3$ are two disjoint cycles in $G$. Hence $N_{G}(p_3)\cap V(C)=\{c_1,c_3\}$. In this case, one can check that $G+p_2c_1+p_2c_3\cong S_{6,3}$. By Observation~\ref{block} (i), $G$ is not $2\mathcal{C}_{\ge 3}$-saturated.\\
Therefore, without loss of generality, we can suppose $N_{G}(p_1)\cap V(C)=\{c_1,c_2\}$. Similarly,  $N_{G}(p_3)\cap V(C)=\{c_j,c_2\}$ for some $j\in\{1,3\}$. If $j=1$, $d_G(c_3)=2$, a contradiction by $\delta(G)\ge 3$. Thus, $j=3$ and we can see that $G\cong W_6$. Note that $W_6$ is $2\mathcal{C}_{\ge 3}$-saturated and $|E(W_6)|=10$ by Observation~\ref{block} (ii).\\
To conclude, the only $6$-vertex $2\mathcal{C}_{\ge 3}$-saturated graph $G$ with $|E(G)|\le 10$ is $W_6$ and $|E(W_6)|=10$. This completes the proof.
\end{proof}
The following lemma state a property of $W_n$.
\begin{lem}\label{wn}
Let $n\ge 6$. If $H$ is a $2\mathcal{C}_{\ge 3}$-saturated graph with $M(H)\cong W_n$, then $H\cong W_n$.
\end{lem}
\begin{proof}
We firstly claim that $M_0(H)\cong W_n$. Notice that $M_0(H)\in\mathcal{S}(M(H))$.
Similarly as the proof of Lemma~\ref{sub}, we only need to prove that for any $uv\in E(G)$ with $G\cong W_n$, $G_0=S_{uv}(G)$ is not $2\mathcal{C}_{\ge 3}$-saturated. Let $X(G)=\{x\}$ and note that $G-X(G)\cong C_{n-1}$. Let $V(G_0)\backslash V(G)=\{w\}$. 
If $x\in \{u,v\}$, say $u=x$, then $xv\notin E(G_0)$ and $G_0+xv\cong W_n^{+1}$. By Observation~\ref{block} (ii), $G_0$ is not $2\mathcal{C}_{\ge 3}$-saturated.
If $x\notin\{u,v\}$, then $xw\notin E(G_0)$ and $G_0+xw\cong W_{n+1}$. By Observation~\ref{block} (ii), $G_0$ is not $2\mathcal{C}_{\ge 3}$-saturated.
Hence, the claim is done.\\
By Lemma~\ref{pre} (i) and the definition of $M_0(H)$, $M_0(H)$ is the only non-trivial block in $H$. Since $M_0(H)\cong W_n$, let $X(M_0(H))=\{x\}$ and pick a vertex $y\in V(M_0(H))\backslash\{x\}$.
If $H\not\cong W_n$, then there is a trivial block, say $B=b_1b_2$, with $V(B)\cap V(M_0(H))\neq\emptyset$. Without loss of generality suppose $b_1\in V(M_0(H))$. If $b_1=x$, then $b_2y\notin E(H)$. One can see that the only non-trivial block of $H+b_2y$ is isomorphic to $W_{n}^{+1}$. By Observation~\ref{block} (ii), a contradiction. If $b_1\neq x$, then $b_2x\notin E(H)$. One can also check that the only non-trivial block of $H+b_2x$ is isomorphic to $W_{n}^{+1}$, a contradiction, too.
\end{proof}
Here is an easy but useful lemma for the proof of Theorem~\ref{sat} and Theorem~\ref{es}.
\begin{lem}\label{x}
Let $G$ be a good $2\mathcal{C}_{\ge3}$-saturated graph on $n\ge 6$ vertices. 
If there exists a vertex $x\in V(G)$ with every cycle in $G$ containing it, then $x\in X(G)$ and $G-\{x\}$ is a tree. In particular, $|E(G)|=2n-3$. Moreover $n\ge 8$.
\end{lem}
\begin{proof}
If there exists $y\in V(G)\backslash\{x\}$ with $xy\notin E(G)$, then by Fact~\ref{f1}, there exist a path $P$ from $x$ to $y$ and a cycle $C$ in $G$ with $V(P)\cap V(C)=\emptyset$. However, $x\in V(P)\cap V(C)$ since every cycle in $G$ contains $x$, which is a contradiction. This makes $x\in X(G)$.\\
Since every cycle in $G$ contains $x$, $T=G-\{x\}$ is a forest. If $T$ contains at least $2$ components, then $x$ is a cut vertex of $G$, which is a contradiction since $G$ is $2$-connected.
Hence $T$ is a tree and $|E(G)|=d_{G}(x)+|E(T)|=(n-1)+(n-2)=2n-3$.\\
Now if $n=6$, then $|E(G)|=9<\sat(6,2\mathcal{C}_\ge 3)$ by Lemma~\ref{n6}, a contradiction.
If $n=7$, let $P$ be a path within $T$ of a maximum size. Let the two endpoints of $P$ be $x$ and $y$. Note that $d_T(x)=d_{T}(y)=1$ and $xy\notin E(G)$. If $T-V(P)$ contains an edge $e=uv$, then $|E(P)|\le |V(T)|-2-1\le 3$ and there is a path $P'$ form $e$ to $V(P)$ since $T$ is connected. Suppose the endpoint of $P'$ in $V(P)$ is $w$, then by pigeonhole principle, one of $P_P(w,x)$ and $P_P(w,y)$, say $P_P(w,x)$ has size at least $\lceil\frac{|E(P)|}{2}\rceil$. Then the path $P_P(w,x)\cup P'\cup \{e\}$ contains at least $\lceil\frac{|E(P)|}{2}\rceil+2>|E(P)|$ edges, which is a contradiction by the maximality of $P$. Therefore, $n\ge 8$.
\end{proof}
To make the proof of Theorem~\ref{sat} brief, we put the following lemma here.
\begin{lem}\label{delete}
Let $G$ be a good $2\mathcal{C}_{\ge3}$-saturated graph with on $n$ vertices with $n\ge 7$ and $|E(G)|\le n+4$. For any vertex $v\in V(G)\backslash D_2(G))$, it holds $|N_{M(G)}(v)\backslash D_2(G))|\ge 3$.
\end{lem}
\begin{proof}
We prove the lemma by contradiction.
If $N_{G}(v)\cap D_2(G)=\emptyset$, then the claim is done since $d_{G}(v)\ge 3$.
So we suppose $u\in N_{G}(v)\cap D_2(G)$. Let $N_{G}(u)=\{v,u_1\}$. Since $n\ge 7>3$, if $d_{G}(u_1)=2$, then $v$ is a cut vertex, which is a contradiction since $G$ is $2$-connected. 
Hence, $u_1\in N_{G}(v)\backslash D_2(G)$.\\
If $N_{G}(v)\backslash D_2(G)=\{u_1\}$, then by the previous proof, any $u'\in  N_{G}(v)\cap D_2(G)$ must have $N_{G}(u')=\{v,u_1\}$. Notice that every cycle in $G$ containing some $u'\in  N_{G}(v)\cap D_2(G)$ must contains $u_1$. So every cycle in $G$ containing $v$ must contain $u_1$ since every cycle containing one of its neighbor contains $u_1$.
Since $G$ is $2\mathcal{C}_{\ge 3}$-saturated, any cycle $C'$ in $G$ must contain one vertex in $\{u,v,u_1\}$ since $uvu_1u$ is a cycle. Thus, $C'$ must contain $u_1$ since if it contains $u$ or $v$, then it contians $u_1$ by the previous proof. Therefore, by Lemma~\ref{x}, $n\ge 8$ and $|E(G)|=2n-3\ge n+5$, a contradiction.\\
Now suppose $N_{G}(v)\backslash D_2(G)=\{u_1, u_2\}$ for some vertex $u_2\in V(G)\backslash D_2(G)$. Similarly as the previous proof, since $vu_1uv$ is a cycle, every cycle in $G$ contains one of $v$ and $u_1$ . Also note that every cycle containing $v$ must contains one of $u_1$ and $u_2$. Otherwise, $v$ is a cut vetex, a contradiction. So every cycle in $G$ contains one of $u_1$ and $u_2$. If $u_1u_2\notin E(G)$, then by Fact~\ref{f1}, there is a path $P$ from $u_1$ to $u_2$ and a cycle $C$ in $G$ with $V(C)\cap V(P)=\emptyset$. However, $C$ contains one of $u_1$ and $u_2$, a contradiction. Thus, $u_1u_2\in E(G)$.\\
If there exists a cycle $C$ that not contains $u_1$, then by the previous proof, $v, u_2\in V(C)$. Suppose $N_{C}(v)=\{u_2, w\}$, then $w\in D_2(G)$ since $N_{G}(v)\backslash D_2(G)=\{u_1, u_2\}$. Let $N_{G}(w)=\{v,u'\}$, then $u'v\in E(G)$ by Proposition~\ref{mg} (iii). Similarly as the previous proof, $u'\notin D_2(G)$. So $u'\in\{u_1,u_2\}$. Since $u_1\notin V(C)$, it holds $u'=u_2$. Now conisider $uu_2\notin E(G)$. By Fact~\ref{f1}, there exist a path $P$ from $u$ to $u_2$ and a cycle $C$ in $G$ with $V(C)\cap V(P)=\emptyset$. Since $N_{G}(u)=\{v,u_1\}$, one of $v$ and $u_1$ is in $V(P)$. If $u_1\in V(P)$, then $\{u_1,u_2\}\subset V(P)$. Recall that every cycle in $G$ contains one of $u_1$ and $u_2$, a contradiction. Thus, $v\in V(P)$, so there exists a cycle $C$ with $v, u_2\notin V(C)$. Also note that every cycle containing $w$ must contain $v$ and $u_2$, so $w\notin V(C)$. Therefore, $C$ and $vwu_2v$ are two disjoint cycles in $G$, a contradiction. Hence every cycle in $C$ contains $u_1$. By Lemma~\ref{x}, $n\ge 8$ and $|E(G)|=2n-3\ge n+5$, a contradiction.
\end{proof}

Now we give the proof of Theorem~\ref{sat}.
\begin{proof}[Proof of Theorem~\ref{sat}]
By Lemma~\ref{n6}, we only prove for $n\ge 7$.
By Theorem~\ref{cal} (i), for $n\ge 7$, $\sat(n,2\mathcal{C}_{\ge3})\le n+5$. Hence, we only need to prove that for $n\ge 7$ and any $n$-vertex $G$ with at most $n+4$ edges, $G$ is not $2\mathcal{C}_{\ge3}$-saturated.
Suppose the contrary that $G$ is an $n$-vertex $2\mathcal{C}_{\ge3}$-saturated graph with at most $n+4$ edges for some $n\ge 7$.
Let $n_0=|V(M(G))|$. By Proposition~\ref{mg} (i), $|E(M(G))|\le n_0+4$.
By Proposition~\ref{mg} (iv), $M(G)\cong K_5$ or $M(G)$ is a $2\mathcal{C}_{\ge3}$-saturated graph.
If $M(G)\cong K_5$, then $|E(M(G))|=10=n_0+5>n_0+4$, a contradiction. Thus, $M(G)$ is $2\mathcal{C}_{\ge3}$-saturated. Morever, if $n_0\le 6$, then $n_0=6$. By Lemma~\ref{n6} and Lemma~\ref{wn}, $G\cong M(G)\cong W_6$, a contradiction by $n\ge 7$. Thus, $n_0\ge 7$. To conclude, $M(G)$ is a $2\mathcal{C}_{\ge3}$-saturated graph on $n_0\ge 7$ vertices with at most $n_0+4$ edges. Also notice that $M(G)$ is good.\\
Suppose $\delta(M(G))\ge 3$. Let $C$ be a shortest cycle in $M(G)$. By Lemma~\ref{pre} (ii), $C\cong C_r$ for some $r\in\{3,4\}$. Let $T=M(G)-V(C)$, then $T$ is a forest by Fact~\ref{f2}. Also, $|V(T)|=n_0-r\ge 3$, and the equality holds if and only if $n_0=7$ and $r=4$. Let $d_i=|D_{i}(T)|$ for any $i\ge 0$ and let $d_{\ge 3}=\sum_{j\ge 3}d_j$.
By Lemma~\ref{pre} (iii), $|E(M(G))|\ge n_0+2d_0+\frac{3}{2}d_1+d_2+\frac{1}{2}d_{\ge 3}$. Since $|E(M(G))|\le n_0+4$, we can see that $d_0+d_1\le 2$. 
Note that every tree contains either a vertex of degree $0$ or at least two vertices of degree $1$. If there are at least two components in $T$, then since $d_0+d_1\le 2$, we have $d_0=2$, $d_1=0$ and $T$ must consist of two components isomorphic to $K_1$. This is a contradiction by $|V(T)|\ge 3$. Thus, $T$ itself is a tree. Since $d_0+d_1\le 2$, it holds $d_1=2$ and then $T$ is a path, This implies $d_2=|V(T)|-2\ge 1$.
Therefore, $n_0+4\ge |E(M(G))|\ge n_0+\frac{3}{2}d_1+d_2=n_0+3+d_1\ge n_0+4$. Hence the equality holds, which means $n_0=7$ and $r=4$. Since $T\cong P_3$ and $C\cong C_4$, suppose $T=p_1p_2p_3$ and $C=c_1c_2c_3c_4c_1$. By Lemma~\ref{pre} (iii), $|N_{M(G)}(p_i)\cap V(C)|=2$ for $i=1, 3$ and $|N_{M(G)}(p_2)\cap V(C)|=1$. Note that $C$ is the shortest cycle in $M(G)$, so there is no copy of $C_3$ in $M(G)$.  This means $N_{M(G)}(p_i)\cap V(C)$ must be either $\{c_1,c_3\}$ or $\{c_2,c_4\}$ for $i=1,3$. Without loss of generality, suppose $N_{M(G)}(p_1)\cap V(C)=\{c_1,c_3\}$. If $N_{M(G)}(p_3)\cap V(C)=\{c_1,c_3\}$, then one of $c_2$ and $c_4$ has degree $2$, a contradiction by $\delta(G)\ge 3$. Thus, $N_{M(G)}(p_3)\cap V(C)=\{c_2,c_4\}$. Now by symmetry, we can suppose $N_{M(G)}(p_2)\cap V(C)=\{c_1\}$ without loss of generality. However, this makes $p_2c_1p_1p_2$ a cycle of size $3$ in $M(G)$, a contradiction by the minimality of $C$. Thus, the proof for $\delta(M(G))\ge 3$ is done. Therefore, it holds $\delta(M(G))=2$. \\
Let $H=M(G)-D_2(M(G))$. By Lemma~\ref{delete}, $\delta(H)\ge 3$, which implies that $|E(H)|\ge \frac{3}{2}|V(H)|=\frac{3}{2}(n_0-t)$, where $t=|D_2(M(G))|$. Then we have
$$E(M(G))=|E(H)|+2|D_2(M(G))|\ge\frac{3}{2}n_0+\frac{1}{2}t\ge (n_0+\frac{7}{2})+\frac{1}{2}=n_0+4\mbox{.}$$
Thus, the equality holds. So $n_0=7$, $t=1$ and any vertex in $H$ has degree $3$. Let $u\in D_2(M(G))$ be the only vertex of degree $2$ in $M(G)$. Let $N_{M(G)}(u)=\{u_1,u_2\}$ and $A=V(M(G))\backslash\{u, u_1, u_2\}$. Since $d_{M(G)}(u)=2$, by Proposition~\ref{mg} (iii), $u_1u_2\in E(M(G))$, which means $uu_1u_2u$ is a cycle in $M(G)$. By Fact~\ref{f2}, $H[A]$ is a forest and then $|E(H[A])|\le(7-1-2)-1=3$. Also, since $d_{H}(u_1)=d_{H}(u_2)=3$ and $u_1u_2\in E(M(G))$, $|N_{H}(u_1)\cap A|=|N_{H}(u_2)\cap A|=2$, which implies $|E(H[\{u_1,u_2\},A])|=4$. Therefore,
$$|E(M(G))|=|E(H[A])|+|E(H[\{u_1,u_2\},A])|+1+2\le 10=n_0+3<n_0+4\mbox{,}$$
a contradiction.
\end{proof}

\section{Saturation spectrum of $2\mathcal{C}_{\ge3}$}
Note that $W_6$, $K_1\vee(K_4\cup K_1)$ and $S_{6,3}$ are $6$-vertex $2\mathcal{C}_{\ge3}$-saturated graph with the number of edges being $10$, $11$ and $12$ respectively. Since $\ES(6,2\mathcal{C}_{\ge3})\subseteq[\sat(n,2\mathcal{C}_{\ge3}),\ex(n,2\mathcal{C}_{\ge3})]=[10,12]$, (i) of Thoerem~\ref{es} follows. So we only need the proof of Thoerem~\ref{es} (ii).
By Theorem~\ref{cal} (ii) and (v), $[n+5,2n-2]\cup\{n+2t: t\in[3,n-3]\}\subseteq\ES(n, 2\mathcal{C}_{\ge3})$ for $n\ge 7$. Then by Theorem~\ref{ex} and Theorem~\ref{sat}, to prove Thoerem~\ref{es} (ii), it remains to show that: for any $n\ge 7$, an $n$-vertex $2\mathcal{C}_{\ge3}$-saturated graph with $m\ge 2n-1$ edges must have $m-n$ even.

\begin{lem}\label{d3}
If $G$ is an $n$-vertex good $2\mathcal{C}_{\ge3}$-saturated graph with $n\ge 7$ and $m=|E(G)|\ge 2n-1$, then $\delta(G)\ge 3$.
\end{lem}
\begin{proof}
Suppose the contrary that there exists a vertex $a\in D_2(G)$. Suppose $N_G(a)=\{b,c\}$. Since $G$ is $2\mathcal{C}_{\ge3}$-saturated and $C^0=abca$ is a cycle, the vertex set of every cycle in $G$ intersects $\{a,b,c\}$. Since every cycle containing $a$ must contains $b$ and $c$, every cycle in $G$ contains at least one of $b$ and $c$.
For any vertices $u,v\in V(G)$, let $\mathcal{C}_G(u,v)$ be the family of any cycle $C$ in $G$ with $u\in V(C)$ and $v\notin V(C)$. We have three claims for $\mathcal{C}_G(b,c)$ and $\mathcal{C}_G(c,b)$.
\begin{claim}\label{bc2}
$|\mathcal{C}_G(b,c)|\ge 2$ and $|\mathcal{C}_G(c,b)|\ge 2$.
\end{claim}
Otherwise, without loss of generality, suppose $|\mathcal{C}_G(b,c)|\le 1$. Let $G'=G$ if $|\mathcal{C}_G(b,c)|=0$ and $G'$ is obtained from $G$ by deleting one edge from the only cycle in $\mathcal{C}_G(b,c)$ if $|\mathcal{C}_G(b,c)|=1$. Since every cycle in $G$ contains at least one of $b$ and $c$, we see that every cycle in $G'$ contains $c$. Thus, $G'-\{c\}$ is a forest and then $|E(G)|\le |E(G')|+1=d_{G'}(c)+|E(G'-\{c\})|+1\le (n-1)+(n-1-1)+1=2n-2<2n-1$, a contradiction. The claim is done.
\begin{claim}\label{bcp}
Let $C\in\mathcal{C}_G(b,c)$ and $C'\in\mathcal{C}_G(c,b)$ be two cycles in $G$, then $C\cap C'$ is a path.
\end{claim}
To prove the claim, note that $C\cap C'\neq\emptyset$, or $C$ and $C'$ are disjoints cycles in $G$, which is a contradiction. Since $C\neq C'$, if $C\cap C'$ is not a path, then $C\cap C'$ is a union of at least two disjoint paths, which implies $|E(C\cap C')|\le |V(C\cap C')|-2$. Notice that $b,c\notin V(C\cap C')$, so $d_{C\cup C'}(b)=d_{C\cup C'}(b)=2$. Therefore, 
\begin{eqnarray*}
|E(C\cup C'-\{b,c\})|&=&|E(C\cup C')|-4=|E(C)|+|E(C')|-|E(C\cap C')|-4\\
&\ge&|V(C)|+|V(C')|-|V(C\cap C')|-2=|V(C\cup C')|-2\\
&=&|V(C\cup C'-\{b,c\})|>|V(C\cup C'-\{b,c\})|-1\mbox{.}
\end{eqnarray*}
This means there is a cycle in $C\cup C'-\{b,c\}\subset G$, which contains neither $b$ nor $c$, a contradiction. The claim is done.
\begin{claim}\label{bbp}
Let $C^1\in\mathcal{C}_G(b,c)$ and $C^2\in\mathcal{C}_G(c,b)$. For any $C\in \mathcal{C}_G(b,c)\cup \mathcal{C}_G(c,b)$ with $C\neq C^1, C^2$, $H=C\cap(C^1\cup C^2-\{b,c\})$ is a path.
\end{claim}
Without loss of generality, suppose $C\in\mathcal{C}_G(b,c)$. 
If $H=\emptyset$, then $C$ is disjoint with $C^2$, a contradiction.
Clearly, $H$ is a union of paths. If $H$ contains at least two components, then $|E(H)|\le |V(H)|-2$.
Let $T=C^1\cup C^2-\{b,c\}$, by Claim~\ref{bcp}, $|E(T)|=|V(T)|-1$. Thereofore, similarly as the proof of Claim~\ref{bcp}, we have
\begin{eqnarray*}
|E((C-\{b\})\cup T)|&=&|E(C-\{b\})|+|E(T)|-|E(H)|\\
&\ge&(|V(C-\{b\})|-1)+(|V(T)|-1)-(|V(H)|-2)\\
&=&|V((C-\{b\})\cup T)|\mbox{.}
\end{eqnarray*}
Thus, there is a cycle in $(C-\{b\})\cup T$,  which contains neither $b$ nor $c$, a contradiction. 
The claim is done.

By Claim~\ref{bc2} and Claim~\ref{bcp}, pick $C^b\in\mathcal{C}_G(b,c)$ and $C^c\in\mathcal{C}_G(c,b)$ with the path $P^0=C^b\cap C^c$ having the largest size. Let $P^0=p_1p_2\dots p_\ell$ for some $\ell\ge 1$. For convience, also let $C^b=bb_1^1b_2^1\dots b_{r_1}^1p_1\dots p_\ell b_{r_2}^2b_{r_2-1}^2\dots b_1^2b$ and $C^c=cc_1^1c_2^1\dots c_{s_1}^1p_1\dots p_\ell c_{s_2}^2c_{s_2-1}^2\dots c_1^2c$ for some $r_1$, $r_2$, $s_1$, $s_2\ge 1$.
Let $\mathcal{P}$ be the family of any path $P\not\subset C^0\cup C^b\cup C^c$ with $V(P)\cap V(C^0\cup C^b\cup C^c)=D_1(P)\neq\emptyset$, which is the set of the two endpoints of $P$.
Since $G$ is $2$-connected, if $\mathcal{P}=\emptyset$, then $G=C^0\cup C^b\cup C^c$, which is a contradiction by Claim~\ref{bc2}.
Let $\mathcal{P}'\subseteq \mathcal{P}$ be the family of path $P\not\subset C^0\cup C^b\cup C^c$ with $V(P)\cap V(C^0\cup C^b\cup C^c)=D_1(P)=\{b,c\}$.
\begin{claim}\label{endp}
Let $u$ and $v$ be the endpoints of some $P\in\mathcal{P}\backslash\mathcal{P}'$, then $|\{u,v\}\cap\{b,c\}|=|\{u,v\}\cap V(P^0)|=1$.
\end{claim}
If $\{u,v\}\cap\{b,c\}=\emptyset$, similarly as the proof of Claim~\ref{bcp} and Claim~\ref{bbp}, one can check that $|E(C^b\cup C^c\cup P-\{b,c\})|=|V(C^b\cup C^c\cup P-\{b,c\})|$, so we can find a cycle which contains neither $b$ nor $c$, a contradiction. Hence by the definition of $\mathcal{P}\backslash\mathcal{P}'$, $|\{u,v\}\cap\{b,c\}|=1$. 
Now without loss of generality, suppose $u=b$ and $v\neq c$.
If $v\in V(C^b)\backslash(\{b\}\cup V(P^0))$, without loss of generality, suppose $v=b_t^1$ for some $t\in[1,r_1]$. Then $P\cup bb_1^1\dots b_1^t$ is a cycle disjoint with $C^c$, a contradiction.
If $v\in V(C^c)\backslash(\{c\}\cup V(P^0))$,  without loss of generality, suppose $v=c_t^1$ for some $t\in[1,s_1]$. Then $C'=P\cup c_t^1\dots c_{s_1}^1 p_1\dots p_\ell b_{r_2}^2\dots b_1^2b$ is a cycle in $\mathcal{C}_G(b,c)$ with $C'\cap C^c=c_t^1\dots c_{s_1}^1 p_1\cup P^0$. This is a contradiction by the maximality of $P_0$. Therefore, $v\in V(P_0)$, and the proof of Claim~\ref{endp} is finished.

We pick $t\in [1,\ell]$ by the following rules.
If $\mathcal{P}\backslash \mathcal{P}'=\emptyset$, we arbitrarily pick $t\in [1,\ell]$.
If $\mathcal{P}\backslash \mathcal{P}'\neq\emptyset$, we arbitrarily pick a path $P^1\in \mathcal{P}\backslash \mathcal{P}'$. Then by Claim~\ref{endp}, without loss of generality, suppose $b\in D_1(P^1)$ and let $t\in[1,\ell]$ be the number with $D_1(P^1)=\{b,p_t\}$.
We give the following key claim.
\begin{claim}\label{cc}
For any cycle $C\in\mathcal{C}_G(b,c)\cup \mathcal{C}_G(c,b)$, it holds $p_t\in V(C)$.
\end{claim}
Suppose the contrary that $p_t\notin V(C)$ for some $C\in C^b\cap C^c$. Clearly, $C\not\subset C^b\cup C^c$. Let $T=C^b\cup C^c-\{b,c\}$.
By Claim~\ref{bbp}, $C\cap T$ is a path. If $p_t\notin V(C\cap T)$, then $C\cap T$ is a path within a disconnected graph $T-\{p_t\}$. One can see that, either $V(C\cap T)\subseteq\{b_1^1,\dots,b_{r_1}^1,c_1^1,\dots,c_{s_1}^1,p_1,\dots,p_{t-1}\}$, or $V(C\cap T)\subseteq\{b_1^2, \dots, b_{r_2}^2, c_1^2, \dots, c_{s_2}^2, p_{t+1}, \dots, p_{\ell}\}$. Without loss of generality, suppose 
\begin{equation}\label{ee1}
V(C\cap T)\subseteq\{b_1^2,\dots,b_{r_2}^2,c_1^2,\dots,c_{s_2}^2,p_{t+1},\dots,p_{\ell}\}\mbox{.}
\end{equation}
If $\mathcal{P}\backslash \mathcal{P}'=\emptyset$, without loss of generality, suppose $C\in\mathcal{C}_G(b,c)$. If $C\cap (C^0\cup C^b\cup C^c)=C\cap (T\cup bb_1^2)$ is path or a union of a path and a vertex. Since $C$ is a cycle, it follows that $C-(D_2(C\cap (C^0\cup C^b\cup C^c))$ is a union of at most two paths. Pick one of it as $P^2$, then $P^2\in\mathcal{P}=\mathcal{P}'$, which is a contradiction since $c\notin P^2\subset C$.
If $P^1\in \mathcal{P}\backslash \mathcal{P}'$ exists,
then $C'=P^1\cup bb_1^1\dots b_{r_1}^1p_1\dots p_{t-1}\in\mathcal{C}_G(b,c)$ is a cycle, so $C\cap C'\neq\emptyset$. Thus, by (\ref{ee1}), $C\cap P^1\neq\emptyset$. Consider $C''=P^1\cup b_1^2\dots b_{r_2}^2p_\ell\dots p_t\in\mathcal{C}_G(b,c)$. Since $C\cap T\neq\emptyset$ is a path, one can see that $C\cap (C''\cup C^c-\{b,c\})$ is disconnected, a contradiction by Claim~\ref{bbp}. 
The claim is done.

Since $ap_t\notin E(G)$, by Fact~\ref{f1}, there is a path $P$ from $a$ to $P_t$ and a cycle $C$ in $G$ with $V(P)\cap V(C)=\emptyset$. Recall that $N_{G}(a)=\{b,c\}$, so $P$ contains at least one of $b$ and $c$. Since every cycle in $G$ also contains at least one of $b$ and $c$, this implies $C\in \mathcal{C}_G(b,c)\cup \mathcal{C}_G(c,b)$ since $V(P)\cap V(C)=\emptyset$. However, by Claim~\ref{cc}, this gives $p_t\in V(C)$. Thus, $p_t\in V(P)\cap V(C)\neq\emptyset$, a contradiction.
\end{proof}
Before the proof of Theorem~\ref{es}, we mention the following result by Dirac~\cite{D60}.
\begin{thm}[Dirac~\cite{D60}]\label{k4}
Any graph $G$ with $\delta(G)\ge 3$ contains a subdivision of $K_4$.
\end{thm}
Now we give a stronger version of Theorem~\ref{es} to finish the proof.
\begin{thm}~\label{stronger}
 Let $G$ be a $2\mathcal{C}_{\ge3}$-saturated graph on $n\ge 7$ vertices with $m\ge 2n-1$ edges. Then $M(G)\cong S_{n_0,3}$ with $n_0=|V(M(G))|\in[7,n]$. In particular, $m-n$ is even.
\end{thm}
\begin{proof}
Let $m_0=|E(M(G))|$. By Proposition~\ref{mg} (i) and Observation~\ref{block} (i), if $M(G)\cong S_{n_0,3}$, then $m-n=m_0-n_0=(3n_0-6)-n_0=2(n_0-3)$, which is even. So it remains to prove that $M(G)\cong S_{n_0,3}$ with $n_0\ge 7$.

Let $G$ be a conterexample that $G$ is a $2\mathcal{C}_{\ge3}$-saturated graph on $n\ge 7$ vertices with $m\ge 2n-1$ and $M(G)\not\cong S_{n_0,3}$.
If $n_0\le 6$, by Proposition~\ref{mg} (iv), Lemma~\ref{n6} and Lemma~\ref{wn}, $M(G)\cong K_5$ or $G\cong W_6$. Since $|V(G)|\ge 7$, we must have $M(G)\cong K_5$. By Proposition~\ref{mg} (i), $m=n+10-5=n+5<2n-1$, a contradiction. Hence $n_0\ge 7$ and then by Proposition~\ref{mg} (iv), $M(G)$ is $2\mathcal{C}_{\ge3}$-saturated.
Note that $m_0=n_0+m-n\ge n_0+2n-1-n=n_0+n-1\ge 2n_0-1$.
By Lemma~\ref{d3} and Theorem~\ref{k4}, we can pick a subgraph $H\subset M(G)$ with $H\in\mathcal{S}(K_4)$. In other words, there exist $4$ vertices $u_i$'s ($i\in[1,4]$) and $6$ internally disjoint paths $P_{j,k}$'s ($1\le j<k\le 4$) with $D_1(P_{j,k})=\{u_j,u_k\}$ for any $1\le j<k\le 4$. Also, $H=\bigcup_{1\le j<k\le 4}P_{j,k}$. We pick such $H$ of the minimum order.
\begin{claim}\label{k4p}
Let $P$ be a path in $M(G)$ with $V(P)\cap V(H)=D_1(P)\neq\emptyset$, then one of the following holds:\\
(a) $D_1(P)\subset\{u_i:i\in[1,4]\}$; \\
(b) $D_1(P)=\{u_i,u'\}$ for some $i\in[1,4]$ and $u'\in D_2(P_{j,k})$ with $j,k\neq i$;\\
(c) $D_1(P)=\{u_1',u_2'\}$ for some $u_i'\in D_2(P_{j_i,k_i})$ ($i=1,2$) with $\{j_1,k_i\}\cap\{j_2,k_2\}=\emptyset$.
\end{claim}
Note that for a path $P$ of size at least $1$, $D_2(P)$ is the set of its internal vertices and $D_1(P)$ is the set of its two endpoints. In particular, $V(P)=D_1(P)\cup D_2(P)$.
If $|D_1(P)\cap\{u_i:i\in[1,4]\}|=2$, then the claim is done by (a).
If $|D_1(P)\cap\{u_i:i\in[1,4]\}|=1$, without loss of generality, suppose $D_1(P)=\{u_1,u'\}$ with $u'\in  D_2(P_{j,k})$ for some $1\le j<k\le4$. If $j=1$, without loss of generality, suppose $k=2$. Then $P\cup P_{P_{1,2}}(u',u_1)$ and $P_{2,3}\cup P_{3,4}\cup P_{2,4}$ are two disjoint cycles in $M(G)$, a contradiction. So $j\neq 1$, which means $j,k\neq 1$. By (b), the proof is done.
Thus, $|D_1(P)\cap\{u_i:i\in[1,4]\}|=0$, suppose $D_1(P)=\{u_1',u_2'\}$ for some $u_i'\in D_2(P_{j_i,k_i})$ ($i=1,2$).
If $|\{j_1,k_i\}\cap\{j_2,k_2\}|=2$, without loss of generality, suppose $j_1=j_2=1$ and $k_1=k_2=2$. Then $P\cup P_{P_{1,2}}(u_1',u_2')$ and $P_{2,3}\cup P_{3,4}\cup P_{2,4}$ are two disjoint cycles in $M(G)$, a contradiction.
If $|\{j_1,k_i\}\cap\{j_2,k_2\}|=1$, without loss of generality, suppose $j_1=j_2=1$, $k_1=2$ and $k_2=3$. Then $P\cup P_{P_{1,2}}(u_1',u_1)\cup P_{P_{1,2}}(u,u_2')$ and $P_{2,3}\cup P_{3,4}\cup P_{2,4}$ are two disjoint cycles in $M(G)$, a contradiction.
Thereofore, $\{j_1,k_i\}\cap\{j_2,k_2\}=\emptyset$, which completes the proof of this claim by (c).
\begin{claim}\label{k43}
If $V(M(G))\neq V(H)$, then $M(G)-V(H)$ is an empty graph. Moreover, for any $u\in V(M(G))\backslash V(H)$, $d_{M(G)}(u)=3$ and $N_{M(G)(u)}=\{u_{j_1},u_{j_2},x\}$, where $x\in V(P_{j_3,j_4})$ and $\{j_i:i\in[1,4]\}=[1,4]$.
\end{claim}
By Fact~\ref{f2}, $M(G)-V(H)$ is a forest. For any component $T$ of $M(G)-V(H)$, let $N_T=\bigcup_{v\in V(M(G))\backslash V(H)}N_{M(G)}(v)$. Note that since $T$ is connected, for any pair $\{a_1,a_2\}\in N_T$, there exists a path $P$ in $M(G)$ with $V(P)\cap V(H)=D_1(P)=\{a_1,a_2\}$. 
If $|N_T\backslash\{u_i:i\in[1,4]\}|\ge 2$, by Claim~\ref{k4p}, we can suppose $u_1',u_2'\in N_T$ with $u_1'\in D_2(P_{1,2})$ and $u_2'\in D_2(P_{3,4})$. However, this leaves no choice for the vertices in $N_T\backslash\{u_1',u_2'\}\neq \emptyset$ in $V(H)$ by Claim~\ref{k4p}, a contradiction.
Hence, $|N_T\backslash\{u_i:i\in[1,4]\}|\le 1$. If $|N_T|=3$ and $|N_T\backslash\{u_i:i\in[1,4]\}|=1$, suppose $N_T=\{u',u_j,u_k\}$, where $\{j,k\}\subset[1,4]$ and $u'\in D_2(P_{j',k'})$ with $\{j',k'\}\subset[1,4]$. One can check by Claim~\ref{k4p} that $\{j',k'\}\cap\{j,k\}=\emptyset$. If $|N_T|=3$ and $|N_T\backslash\{u_i:i\in[1,4]\}|=0$, clearly, $N_T=\{u_i,u_j,u_k\}$ with $\{i,j,k\}\in [1,4]$. Thus, we can conclude that:
\begin{equation}\label{nt3}
\mbox{If }|N_T|= 3\mbox{, then }N_T=\{u_{j_1},u_{j_2},x\}\mbox{, where }x\in V(P_{j_3,j_4})\mbox{ and }\{j_i:i\in[1,4]\}=[1,4]\mbox{.}
\end{equation}
If $|N_T|\ge 4$, then $|N_T\cap\{u_i:i\in[1,4]\}|\ge 4-1=3$. Without loss of generality, suppose $u_1,u_2,u_3\in N_T$. Let $u'\in N_T\backslash\{u_1,u_2,u_3\}\neq\emptyset$, then one can check by Claim~\ref{k4p} that $u'=u_4$. Thus, we must have $|N_T|=4$ and $N_T=\{u_i:i\in[1,4]\}$.
Therefore, we have the following result:
\begin{equation}\label{nt4}
\mbox{If }|N_T|\ge 4\mbox{, then }|N_T|=4\mbox{ and }N_T=\{u_i:i\in[1,4]\}\mbox{.}
\end{equation}
Now suppose $|N_T|\ge 4$
If $T\not\cong K_1$, then since $T$ is a tree, $|D_1(T)|\ge 2$.
Pick $t_1,t_2\in D_1(T)$ and let $N_i=N_{M(G)}(t_i)\cap V(H)$ for $i=1,2$. Clearly, by Lemma~\ref{d3}, $|N_i|\ge 2$ for $i=1,2$. If there exists some $v\in N_1\cap N_2$, without loss of generality, suppose $v=u_1$ if $v\in\{u_i:i\in[1,4]\}$ and suppose $v\in D_2(P_{1,2})$ otherwise.
Then $C=t_1vt_2\cup P_{T}(t_1,t_2)$ is a cycle in $M(G)$ disjoint with the cycle $P_{2,3}\cup P_{3,4}\cup P_{2,4}$, a contradiction. So $ N_1\cap N_2=\emptyset$. Hence, $|N_1\cup N_2|\ge 4$. Clearly, $N_1\cup N_2\subseteq N_T$, so $|N_T|\ge 4$. By (\ref{nt4}), this means $|N_T|=4$ and $N_T=\{u_i:i\in[1,4]\}$. Since the equality holds, we get $|N_1|=|N_2|=2$. Without loss of generality, suppose $N_1=\{u_1,u_2\}$ and $N_2=\{u_3,u_4\}$. Then $t_1u_1u_2t_1$ and $t_2u_3u_4t_2$ are two disjoint cycles in $M(G)$, a contradiction. 
Therefore, $T\cong K_1$ for any component in $M(G)-V(H)$, i.e. $M(G)-V(H)$ is an empty graph. let $V(T)=\{t\}$. Suppose $d_{M(G)}(t)\ge 4$. Since $N_T=N_{M(G)}(t)$, we get $|N_T|\ge 4$, which implies $|N_T|=4$ and $N_T=\{u_i:i\in[1,4]\}$.
Suppose there exists some $P_{j,k}$ for $1\le j<k\le 4$ with $E(P_{j,k})\ge 2$. Without loss of generality, suppose $E(P_{1,2})\ge 2$. Then one can see that $H-(\{u_1\}\cup\bigcup_{k=2}^4D_2(P_{1,k}))+\{tu_2,tu_3,tu_4\}$ is a subdivision of $K_4$ with a smaller order than $H$, a contradiction by the minimality of $H$. Therefore, $|E(P_{j,k})|=1$ for $1\le j<k\le 4$, so $H\cong K_4$ and $M(G)[\{t\}\cup V(H)]\cong K_5$. Since $n_0\ge 7>5$, we can pick a vertex $v\in V(M(G))\backslash (V(H)\cup\{t\})$. Similarly, $v$ itself is a component of $M(G)-V(H)$, so $v$ has at least three neighbors in $V(H)$. Without loss of generality, suppose $u_1,u_2\in N_{M(G)}(v)$. Then $vu_1u_2v$ and $tu_3u_4t$ are two disjoint cycles in $M(G)$, a contradiction.
Therefore, by Lemma~\ref{d3}, $d_{M(G)}(t)=3$ for any $t\in M(G)-V(H)$. Then by (\ref{nt3}), the claim follows.

Let $A=E(M(G)[V(H)])\backslash E(H)$ be a set of edges.\\
\textbf{Case 1.} $A=\emptyset$.\\
In this case $H=M(G)[V(H)]$. Note that $H\in\mathcal{S}(K_4)$. 
If $H\cong K_4$, then since $n_0-4\ge 3$, we can pick a vertex $v_0\in V(M(G))\backslash V(H)$. By Claim~\ref{k43}, without loss of generality, suppose $N_{M(G)}(v_0)=\{u_1,u_2,u_3\}$. Then for any $v\in V(M(G))\backslash (V(H)\cup\{v_0\})$, if $N_{M(G)}(v)\neq\{u_1,u_2,u_3\}$, without loss of generality, suppose $N_{M(G)}(v)\neq\{u_1,u_2,u_4\}$. Then $v_0u_1u_3v_0$ and $vu_2u_4v$ are two disjoint cycles, a contradiction. Hence for any $v\in V(M(G))\backslash V(H)$, it holds $N_{M(G)}(v)=\{u_1,u_2,u_3\}$. Thus, $M(G)=K[\{u_1,u_2,u_3\},V(M(G))\backslash\{u_1,u_2,u_3\}]\cup u_1u_2u_3u_1\cong S_{n_0,3}$ and we are done.\\
If $H\not\cong K_4$, then $V(H)\backslash\{u_i:i\in[1,4]\}\neq\emptyset$. Without loss of generality, supoose there is a vertex $u_0\in D_2(P_{1,2})$. Since $d_{M(G)}(u_0)\ge 3$ by Lemma~\ref{d3} and $d_{H}(u_0)=2$, there exists a vertex $v_0\in V(M(G))\backslash V(H)$ with $u_0\in N_{M(G)}(v_0)$. Then by Claim~\ref{k43}, $N_{M(G)}(v_0)=\{u_0,u_3,u_4\}$.
For any vertex $v\in V(M(G))\backslash (V(H)\cup\{v_0\}$, by Claim~\ref{k43}, $|N_{M(G)}(v)\cap\{u_i:i\in[1,4]\}|\ge 2$. Suppose $\{u_{j_1},u_{j_2}\}\subset N_{M(G)}(v)$ for some $\{j_1,j_2\}\subset [1,4]$. If $\{j_1,j_2\}=\{1,2\}$, then $vu_1u_2v$ and $v_0u_3u_4v_0$ are two disjoint cycles, a contradiction. If $|\{j_1,j_2\}\cap\{1,2\}|=1$, without loss of generality, suppose $\{j_1,j_2\}=\{1,3\}$. Then $vu_1u_3v$ and $u_0v_0u_4\cup P_{2,4}\cup P_{P_{1,2}}(u_2,u_0)$ are two disjoint cycles, a contradiction. Hence, $\{j_3,j_4\}=\{3,4\}$. Then by Claim~\ref{k43}, $N_{G}(v)=\{u_0',u_3,u_4\}$, where $u_0'\in V(P_{1,2})$. If $u_0'\in V(P_{P_{1,2}}(u_1,u_0))\backslash \{u_0\}$, then $u_0'vu_3\cup P_{1,3}\cup P_{P_{1,2}}(u_1,u_0')$ and $u_0v_0u_4\cup P_{2,4}\cup P_{P_{1,2}}(u_2,u_0)$ are two disjoint cycles, a contradiction. So $u_0'\notin V(P_{P_{1,2}}(u_1,u_0))\backslash \{u_0\}$. Similarly, $u_0'\notin V(P_{P_{1,2}}(u_2,u_0))\backslash \{u_0\}$. Thus, $u_0'=u_0$. To conclude, for any $v\in V(M(G))\backslash V(H)$, $N_{M(G)}(v)=\{u_0,u_3,u_4\}$. Now, we can see that $M(G)=K[\{u_0,u_3,u_4\},V(M(G))\backslash\{u_0,u_3,u_4\}]+u_3u_4$. Hence $M(G)+u_0u_3+u_0u_4\cong S_{n_0,3}$. By Observation~\ref{block} (i), this is a contradiction.
\\
\textbf{Case 2.} There is an edge $e\in A$ with $e\subset\{u_i:i\in[1,4]\}$.\\
Let $e=u_iu_j$ for some $1\le i<j\le 4$. Then $P_{i,j}$ is a path of size at least $2$. Then $(E(H)\backslash E(P_{i,j}))\cup\{e\}$ induce a subdivision of $K_4$ with a smaller order than $H$, a contradiction by the minimality of $H$.\\ 
\textbf{Case 3.} There is an edge $e\in A$ with $e\cap\{u_i:i\in[1,4]\}=\emptyset$.\\
By Claim~\ref{k4p}, we can suppose $e=v_1v_2$ with $v_1\in D_2(P_{1,3})$ and $v_2\in D_2(P_{2,4})$ without loss of generality. 
One can check that, $H+e-D_2(P)$ is a subdivision of $K_4$ if $|E(P)|\ge 2$ and
$$P\in\mathcal{P}_1=\{P_{1,2},P_{1,4},P_{2,3},P_{3,4},P_{P_{1,3}}(v_1,u_1), P_{P_{1,3}}(v_1,u_3),P_{P_{2,4}}(v_2,u_2),P_{P_{2,4}}(v_2,u_4)\}\mbox{.}$$
Hence, any path $P\in\mathcal{P}_1$ contains only one edge, or $H+e-D_2(P)$ is a smaller subdivision of $K_4$ than $H$, which is a contradiction by the minimality of $H$. Therefore, $V(H)=\{u_i:i\in[1,4]\}\cup\{v_1,v_2\}$ and one can see that $H+e=K[\{u_1,u_3,v_2\},\{u_2,u_4,v_1\}]\cong K_{3,3}$.
Since $|V(H)|=6<|M(G)|$, it holds $V(M(G))\backslash V(H)\neq\emptyset$. Pick a vertex $v_0\in V(M(G))\backslash V(H)$. By Claim~\ref{k43}, $|N_{M(G)}(v_0)\cup V(H)|=|N_{M(G)}(v_0)|=3$. If some edge $e'\in E(H+e)$ has $e'\subset N_{M(G)}(v_0)$, then $e'\cup v_0$ induce a copy of $C_3$ and $V(H)\backslash e'$ induce a copy of $C_4$, which is a contradiction since $M(G)$ is $2\mathcal{C}_{\ge 3}$-saturated. Therefore, either $N_{M(G)}(v_0)=\{u_1,u_3,v_2\}$ or $N_{M(G)}(v_0)=\{u_2,u_4,v_1\}$. Without loss of generality, suppose $N_{M(G)}(v_0)=\{u_1,u_3,v_2\}$. For any vertex $v\in V(M(G))\backslash (V(H)\cup\{v_0\})$, if $N_{M(G)}(v)=\{u_2,u_4,v_1\}$, then $v_0u_1u_2u_3v_0$ and $vv_1v_2u_4v$ are disjoint copies of $C_4$, which is a contradiction. Thus, by the previous proof, for any $v\in V(M(G))\backslash V(H)$, it always holds $N_{M(G)}(v)=\{u_1,u_3,v_2\}$. Hence $K[\{u_1,u_3,v_2\},V(M(G))\backslash\{u_1,u_3,v_2\}]\subset M(G)$. Now since $M(G)$ is $2\mathcal{C}_{\ge 3}$-saturated and $|V(M(G))|\ge 7$, one can check that any pair in $V(M(G))\backslash\{u_1,u_3,v_2\}$ is not an edge and then any pair in $\{u_1,u_3,v_2\}$ is an edge. Therefore, $M(G)\cong S_{n_0,3}$.\\
\textbf{Case 4.} Any edge $e\in A$ has $|e\cap\{u_i:i\in[1,4]\}|=1$.\\
By Claim~\ref{k4p}, we can suppose $e=u_1v_1$ with $v_1\in D_2(P_{2,3})$ without loss of generality.
One can check that, $H+e-D_2(P)$ is a subdivision of $K_4$ if $|E(P)|\ge 2$ and
$$P\in\mathcal{P}_2=\{P_{1,i}: i\in[2,4]\}\mbox{.}$$
Similarly as Case 3, any path $P\in\mathcal{P}_2$ contains only one edge.
\begin{claim}\label{u1}
For any edge $e'\in A$, it holds $u_1\in e'$.
\end{claim}
Suppose the contrary that there is another $e'\in A$ with $u_1\notin e'$. 
If $u_2\in e'$, by Claim~\ref{k4p} and $D_2(P_{1,3})=\emptyset$, it holds $e'=u_2v_2$ with $v_2\in D_2(P_{1,4})\cup D_2(P_{3,4})$. If $v_2\in D_2(P_{1,4})$, then $\{e'\}\cup P_{P_{1,4}}(v_2,u_4)\cup P_{2,4}$ and $\{e\}\cup P_{2,3}(v_1,u_3)\cup P_{1,3}$ are two disjoint cycles in $M(G)$, a contradiction. If $v_2\in D_2(P_{3,4})$, then $\{e'\}\cup P_{P_{3,4}}(v_2,u_4)\cup P_{2,4}$ and $\{e\}\cup P_{2,3}(v_1,u_3)\cup P_{1,3}$ are two disjoint cycles in $M(G)$, a contradiction, too. Thus, $u_2\not\in e'$. Similarly,  $u_3\not\in e'$, so $u_4\in e'$. Since $D_2(P_{1,2})=D_2(P_{1,3})=\emptyset$, it holds $e'=u_4v_2$ with $v_2\in D_2(P_{2,3})$.
Similarly we can prove that each of $P_{3,4}$ and $P_{2,4}$ contains only one edge. Now by symmetry, suppose $v_2\in V(P_{P_{2,3}}(v_1,u_2))$, then $H-(\{u_3\}\cup D_2(P_{P_{2,3}}(u_3,v_1)))$ is a subdivision of $K_4$, which is a contradiction by the minimality of $H$. The claim is done.

By Claim~\ref{u1}, we see $M(G)[V(H)]-\{u_1\}\cong C_{|V(H)|-1}$, so $M(G)[V(H)]$ is a subgraph of a copy of $W_{|V(H)|}$. Hence, by Observation~\ref{block} (ii), if $V(H)=V(M(G))$, then $m_0\le 2n_0-2$, a contradiction. Thus, $V(M(G))\backslash V(H)\neq\emptyset$.
\begin{claim}\label{wheel}
For any $v\in V(M(G))\backslash V(H)$ and $a_1,a_2\in N_{M(G)}(v)\backslash\{u_1\}$, there is no path $P$ in $M(G)[H]-\{u_1,a_1,a_2\}$ with $\emptyset\neq D_1(P)\subset N_{M(G)}(u_1)$.
\end{claim}
If such path $P$ exists, suppose the endpoints of $P$ are $p_1$ and $p_2$.
Note that $M(G)[V(H)]-\{u_1\}\cup V(P)$ is a path, which is connected. Hence, $a_1va_2\cup P_{M(G)[V(H)]-(\{u_1\}\cup V(P)}(a_1,a_2)$ is a cycle disjoint with the cycle $P\cup p_1u_1p_2$, a contradiction. The claim is done.

Now suppose $d_{M(G)[V(H)]}(u_1)=d$. Recall that $\{u_2,u_3,u_4,v_1\}\subseteq N_{M(G)[V(H)]}(u_1)$, so $d\ge 4$.
By Claim~\ref{k43}, for any $v\in V(M(G))\backslash V(H)$, we have $|N_{M(G)}(v)\backslash\{u_1\}|\ge 2$. So we can pick $a_1,a_2\in N_{M(G)}(v)\backslash\{u_1\}$.
Note that $M(G)[H]-\{u_1,a_1,a_2\}$ is a union of at most two disjoint paths and there are at least $d-2$ neighbors of $u_1$ in it. By piegeonhole principle, if $d\ge 5$, then there must exist at least $\lceil\frac{d-2}{2}\rceil=2$ of them in one component of $M(G)[H]-\{u_1,a_1,a_2\}$, which means they are the endpoints of some path in $M(G)[H]-\{u_1,a_1,a_2\}$. By Claim~\ref{wheel}, this is a contradiction.
Thus $d=4$, so $N_{M(G)[V(H)]}(u_1)=\{u_2,u_3,u_4,v_1\}$. By the previous proof, the only case can happen now is $N_{M(G)}(v)\backslash\{u_1\}=\{a_1,a_2\}\subset N_{M(G)}(u_1)$ and the two neighbors of $u_1$ within $V(H)$ other than $a_1$ and $a_2$ are in different components of $M(G)[H]-\{u_1,a_1,a_2\}$. Now it is easy to see that $N_{M(G)}(v)=\{u_1,u_2,u_3\}$ or $N_{M(G)}(v)=\{u_1,v_1,u_4\}$.
Pick a vertex $v'\in V(M(G))\backslash V(H)$, without loss of generality, suppose $N_{M(G)}(v')=\{u_1,u_2,u_3\}$. If there exists some $v''\in V(M(G))\backslash (V(H)\cup\{v'\}$ with $N_{M(G)}(v'')=\{u_1,v_1,u_4\}$, then $v'u_1u_2v'$ and $v''v_1u_3u_4v''$ are two disjoint cycles in $M(G)$, a contradiction. This means for any $v\in V(M(G))\backslash V(H)$, we always have $N_{M(G)}(v)=\{u_1,u_2,u_3\}$. Now one can check that $K[\{u_1,u_2,u_3\},V(M(G))\backslash\{u_1,u_2,u_3\}]+u_1u_2+u_1u_3\subset M(G)$. Similarly as the proof in Case 3, it follows that $M(G)\cong S_{n_0,3}$ and the proof is done.
\end{proof}

\section{Concluding remarks}
In this article, we completely determine the saturation number and the saturation spectrum of $k\mathcal{C}_{\ge 3}$ when $k=2$ and give some results when $k\ge 3$. 
Below we give some remarks and conjectures for $\sat(n,k\mathcal{C}_{\ge3})$ and $\ES(n,k\mathcal{C}_{\ge 3})$ when $k\ge 3$.

Firstly, by Theorem~\ref{k} (i), $\sat(n, k\mathcal{C}_{\ge 3})\le n+6k-7$ for $n\ge 4k-3$ and $k\ge 3$. It is natural to make the following conjecture.
\begin{conj}\label{conj1}
For $n\ge 4k-3$ and $k\ge 3$, $\sat(n, k\mathcal{C}_{\ge 3})=n+6k-7$.
\end{conj}
For $3\mathcal{C}_{\ge 3}$,
by Theorem~\ref{cal} (iii) and (v) that, $[n+11,3n-12]\cup\{n+1+4t: t\in[5,n-4]\}\subseteq\ES(n,3\mathcal{C}_{\ge 3})$ for $n\ge 14$. 
On may conjecture the equality holds.
However, this is not ture.
For example, when $n\ge 12$, by considering a graph with $s_{6,2}=s_{n-5,2}=1$ and another graph with $s_{6,2}=s_{n-6,2}=k_1=1$ in Claim~\ref{construct}, 
we can construct $n$-vertex $3\mathcal{C}_{\ge 3}$-saturated graphs with $3n-9$ and $3n-11$ edges respectively.
Moreover, for $n\ge t+4\ge 5$, let $W_{n,t}\cong K_{t}\vee C_{n-t}$.
For $n\ge 3t+5$ and $t\ge 3$, let $T_{n,t}=K_t\vee U(P^1,\dots, P^{t+2}, v_1,\dots, v_{t+2})$, where $P^i\cong P_3$ for $i\in[1,t+1]$, $P^{t+2}\cong P_{n-3t-2}$, and $v_j\in D_1(P^j)$ for any $j\in[1,t+2]$.
One can check the following observations.
\begin{obs}\label{other}
(i) For $k\ge 3$ and $n\ge 5k-4$, $W_{n,k-1}$ is an $n$-vertex $k\mathcal{C}_{\ge 3}$-saturated graphs with $\binom{k-1}{2}+k(n-k+1)$ edges.\\
(ii) For $k\ge 2$ and $n\ge 3k+2$, $T_{n,k-1}$ is an $n$-vertex $k\mathcal{C}_{\ge 3}$-saturated graphs with $\binom{k-1}{2}+k(n-k+1)-1$ edges.
\end{obs}
By considering these mentioned constructions, we give the following conjecture.
\begin{conj}\label{conj2}
For $n\ge 14$, 
$$\ES(n,3\mathcal{C}_{\ge 3})=([n+11,3n-11]\cup\{3n-9,3n-6,3n-5\})\cup\{n+1+4t: t\in[5,n-4]\}\mbox{.}$$ 
\end{conj}

More generally, we ask the following problem at last.
\begin{prob}\label{prob1}
Determine $\ES(n,k\mathcal{C}_{\ge 3})$ for $k\ge 3$ and $n\ge 3k$.
\end{prob}

\end{document}